\documentclass[11pt]{article}

\usepackage{enumerate}
\usepackage{amsmath}
\usepackage{amssymb,latexsym}
\usepackage{amsthm}
\usepackage{color}
\usepackage{cancel}
\usepackage{graphicx}

\usepackage{tikz}

\usepackage{hyperref}

\usepackage{amscd}

\DeclareMathOperator{\C}{\mathcal{C}}

\newtheorem{theorem}{Theorem}[section]

\newtheorem{lemma}[theorem]{Lemma}
\newtheorem{corollary}[theorem]{Corollary}
\newtheorem{definition}[theorem]{Definition}
\newtheorem{proposition}[theorem]{Proposition}

\newtheorem{conjecture}[theorem]{Conjecture}

\newtheorem{remark}[theorem]{Remark}

\newcommand{\cL}{{\mathcal L}}

\newcommand{\F}{{\mathbb F}}

\newcommand{\fq}{{\mathbb F}_{q}}

\newcommand{\N}{\mathrm{N}}

\title{Linearized trinomials with maximum kernel}

\author{Paolo Santonastaso and Ferdinando Zullo}
\date{}

\begin{document}
\maketitle

\begin{abstract}
Linearized polynomials have attracted a lot of attention because of their applications in both geometric and algebraic areas.
Let $q$ be a prime power, $n$ be a positive integer and $\sigma$ be a generator of $\mathrm{Gal}(\F_{q^n}\colon\F_q)$.
In this paper we provide closed formulas for the coefficients of a $\sigma$-trinomial $f$ over $\F_{q^n}$ which ensure that the dimension of the kernel of $f$ equals its $\sigma$-degree, that is \emph{linearized polynomials with maximum kernel}.
As a consequence, we present explicit examples of linearized trinomials with maximum kernel and characterize those having $\sigma$-degree $3$ and $4$. Our techniques rely on the tools developed in \cite{McGuireMuller}.
Finally, we apply these results to investigate a class of rank metric codes introduced in \cite{CsMPZh}, to construct quasi-subfield polynomials and cyclic subspace codes, obtaining new explicit constructions to the conjecture posed in \cite{Trautmann}.
\end{abstract}



\section{Introduction}

Let $q=p^h$, where $h$ is a positive integer and $p$ is a prime.
Linearized polynomials over $\F_{q^n}$
correspond to $\F_q$-linear transformations of the $n$-dimensional $\F_q$-vector space $\F_{q^n}$. 
For this reason they have been intensively investigated and used to describe geometric and algebraic objects such as rank metric codes, $\F_q$-linear sets, see e.g.\ \cite{BM,BZ,qres,teoremone,McGuireSheekey,Sihem,PZ2019,John,Zanella}.

A $\sigma$-\emph{linearized polynomial} (for short $\sigma$-\emph{polynomial}) over $\F_{q^n}$ is a polynomial of shape
\[f(x)=\sum_{i=0}^t a_i x^{\sigma^i},\]
where $a_i\in \F_{q^n}$, $t$ is a positive integer and $\sigma$ is a generator of the Galois group $\mathrm{Gal}(\F_{q^n}\colon \F_q)$.
Furthermore, if $a_t \neq 0$ we say that $t$ is the $\sigma$-\emph{degree} of $f(x)$.
We will denote by $\mathcal{L}_{n,q,\sigma}$ the set of all $\sigma$-polynomials over $\F_{q^n}$ (or simply by $\mathcal{L}_{n,q}$ if $\sigma\colon x \in \F_{q^n}\mapsto x^q\in \F_{q^n}$).
In the remainder of this paper we shall
always silently identify the elements of $\cL_{n,q,\sigma}$ with the
endomorphisms of $\F_{q^n}$ they represent and, as such,
speak also of \emph{kernel} and \emph{rank} of a polynomial.
Clearly, the kernel of $f(x)\in \cL_{n,q,\sigma}$ coincides with the set of the roots of $f(x)$ over $\F_{q^n}$.

As a consequence of \cite[Theorem 5]{GQ2009} and \cite[Theorem 10]{GQ2009x}, for $\sigma$-linearized polynomials we have the following bound on the number of roots. 

\begin{theorem}\label{Gow}
Consider
\[f(x)=a_0x+a_1x^{\sigma}+\cdots+a_{k-1}x^{\sigma^{k-1}}+a_kx^{\sigma^k}\in \mathcal{L}_{n,q,\sigma},\]
with $k\leq n-1$ and let $a_0,a_1,\ldots,a_k$ be elements of $\F_{q^n}$ not all of them zero. Then 
\[ \dim_{\F_q}(\ker (f(x)))\leq k. \]
Furthermore, if $\dim_{\F_q}(\ker (f(x)))=k$ then $\N_{q^n/q}(a_0)=(-1)^{nk}\N_{q^n/q}(a_k)$.
\end{theorem}

Particular attention was paid for those polynomials attaining the bound of Theorem \ref{Gow}, which are called \emph{linearized polynomials with maximum kernel} or \emph{subspace polynomials} when $\sigma\colon x \in \F_{q^n}\mapsto x^q\in \F_{q^n}$.
The coefficients of a $\sigma$-linearized polynomials have been characterized as follows.

\begin{theorem}(\cite[Theorem 1.2 and Corollary 3.2]{teoremone},\cite[Theorem 7]{McGuireSheekey})\label{th:maxker}
Consider
\[f(x)=a_0x+a_1x^{\sigma}+\cdots+a_{k-1}x^{\sigma^{k-1}}-x^{\sigma^k}\in \mathcal{L}_{n,q,\sigma},\]
with $k\leq n-1$.
Then $f(x)$ has maximum kernel if and only if
\begin{equation}\label{eq:Cfrel}
C_f C_f^{\sigma} \cdot \ldots \cdot C_f^{\sigma^{n-1}}=I_k,
\end{equation}
where 
\[C_f=\left( \begin{matrix} 
0 & 0 & \cdots & 0 & a_0 \\
1 & 0 & \cdots & 0 & a_1 \\
0 & 1 & \cdots & 0 & a_2 \\
\vdots & \vdots & \ddots & \vdots & \vdots \\
0 & 0 & \cdots & 1 & a_{k-1} 
\end{matrix}
\right)\] 
and is called the $\sigma$-companion matrix of $f$, $C_f^{\sigma^i}$ is the matrix obtained from $C_f$ by applying to each of its entries the automorphism $x\mapsto x^{\sigma^i}$ and $I_k$ is the identity matrix of order $k$.
Furthermore, Equation \eqref{eq:Cfrel} holds if and only if
\begin{equation}\label{eq:Cfrelshort}
C_f C_f^{\sigma} \cdot \ldots \cdot C_f^{\sigma^{n-1}}\mathbf{e}_0=\mathbf{e}_0,
\end{equation}
where $\mathbf{e}_0=(1,0,\ldots,0)^T\in \mathbb{F}_{q^n}^{k\times 1}$.
\end{theorem}

Linearized polynomials attracted a lot of attention because of their applications to construct MRD codes \cite{John}, to study list decodability of rank metric codes \cite{raviv_2016,TrombZullo,wachter-zhe_2013}, cyclic subspace codes \cite{BEGR,Otal} and affine dispersers \cite{BK} and Elliptic curve discrete logarithm problem \cite{HKP}. 
Motivated by these applications, we further investigated the results on linearized trinomials with maximum kernel of \cite{McGuireMuller}. 
More precisely, we first refine their results for $\sigma$-polynomials proving the following.

\begin{theorem}\label{th:mainMcG}
Let $d,n$ and $s$ be positive integers with $\gcd(s,n)=1$ such that $d\geq 3$. Let denote by $\sigma$ the automorphism $x\in \mathbb{F}_{q^n}\mapsto x^{q^s}\in \mathbb{F}_{q^n}$.
\begin{enumerate}
    \item[(a)] If $n\leq d(d-1)$ and $d\nmid n$, then $\dim_{\mathbb{F}_q}(\ker(ax+bx^{\sigma}-x^{\sigma^d}))<d$, for any $a,b \in \mathbb{F}_{q^n}$.
    \item[(b)] If $n\leq d(d-1)$ and $d\mid n$, then $\dim_{\mathbb{F}_q}(\ker(ax+bx^{\sigma}-x^{\sigma^d}))=d$ with $a,b \in \mathbb{F}_{q^n}$ if and only if $b=0$ and $\N_{q^n/q^d}(a)=1$.
    \item[(c)] If $n=d(d-1)+1$, then $\dim_{\mathbb{F}_q}(\ker(ax+bx^{\sigma}-x^{\sigma^d}))=d$ with $a,b \in \mathbb{F}_{q^n}$ if and only if 
    \[ \left\{ \begin{array}{lll} \mathrm{N}_{q^n/q}(a)=(-1)^{d-1}, \\ b=-a^{\sigma f_1}\,\,\text{where}\,\, f_1=\sum_{i=0}^{d-1}\sigma^{id},\\ 
    d-1\,\, \text{is a power of}\,\, p.
    \end{array}\right. \]
\end{enumerate}
\end{theorem}

Then we characterize $\sigma$-trinomials with maximum kernel giving closed formulas involving their coefficients. 

\begin{theorem}\label{th:main}
Let $d,n$ and $s$ be positive integers with $\gcd(s,n)=1$ such that $d\geq 3$ and $n=d(d-1)+g$, for some $g\in \{1,\ldots,d-1\}$. Let denote by $\sigma$ the automorphism $x\in \mathbb{F}_{q^n}\mapsto x^{q^s}\in \mathbb{F}_{q^n}$.
Then $\dim_{\mathbb{F}_q}(\ker(ax+bx^{\sigma}-x^{\sigma^d}))=d$ with $a,b \in \mathbb{F}_{q^n}$ if and only if $a$ and $b$ satisfy the following system
    $$
\sum_{\substack{i_0,\ldots,i_{t}=0 \\ i_0+\ldots+i_t=d-1-t}}^{d-1-t}  \left( \prod_{j_0=1}^{i_0} a^{\sigma^{d(d-1-j_0)+g}} \right) \left( \prod_{j_1=1}^{i_1} a^{\sigma^{d(d-i_0-j_1-2)+g+1}} \right)\cdot \ldots \cdot $$
$$\left( \prod_{j_{t-1}=1}^{i_{t-1}} a^{\sigma^{d(d-i_0-\ldots-i_{t-2} -j_{t-1} -t)+g+t-1}} \right) \cdot \left( \prod_{j_{t}=1}^{i_t} a^{\sigma^{d(i_t-j_t)+t+g}} \right) \cdot$$ $$\left( \prod_{u=0}^{t-1} b^{\sigma^{d(d-2-i_0-\cdots-i_u-u)+u+g}} \right)$$
$$=\left\{ \begin{array}{lll} (-1)^v b^{\frac{\sigma^{v}-1}{\sigma-1}} a^{-\frac{\sigma^{v+1}-1}{\sigma-1}} & \text{if}\,\, t=d-g+v, \,\, 0\leq v\leq 
g-1,\\
(-1)^g a^{-\frac{\sigma^{g}-1}{\sigma-1}}b^{\frac{\sigma^{g}-1}{\sigma-1}}& \text{if}\,\, t=0,\\ 
0 & \text{otherwise},\end{array}\right. 
$$
\end{theorem}

As a byproduct, in Corollary \ref{cor:neccond} we also get necessary conditions on the coefficients of the trinomial to have maximum kernel.
Then we use Theorem \ref{th:main} to introduce a new class of $\sigma$-linearized trinomials with maximum kernel when $q$ is even and we determine such equations when $d=3$ and $d=4$.
Finally we apply these results to investigate rank metric codes introduced in \cite{CsMPZh}, to construct examples of quasi-subfield polynomials and of cyclic subspace codes. These constructions of subspace codes yield new solutions to the conjecture posed in \cite{Trautmann}.

The paper is organized as follows. Section \ref{sec:thmainMcG} regards the proof of Theorem \ref{th:mainMcG}.
Section \ref{sec:thmain} is devoted to the proof of Theorem \ref{th:main} and in Section \ref{sec:neccond} we pointed out some necessary conditions on the coefficients of the trinomial to have maximum kernel. In Section \ref{sec:newfamily} we apply Theorem \ref{th:main} to obtain a new family of $\sigma$-linearized trinomials with maximum kernel.
In Section \ref{sec:d=3,4} we give explicit conditions on the coefficient of a $\sigma$-linearized trinomials when the $\sigma$-degree is either $3$ or $4$.
Section \ref{sec:applications} deals with the applications in rank metric codes, quasi-subfield polynomials and cyclic subspace codes.
We conclude the paper listing some possible open problems in Section \ref{sec:final}.

\section{Proof of Theorem \ref{th:mainMcG}}\label{sec:thmainMcG}

Let $d,n$ and $s$ be positive integers with $\gcd(s,n)=1$ such that $d\geq 3$. Let denote by $\sigma$ the automorphism $x\in \mathbb{F}_{q^n}\mapsto x^{q^s}\in \mathbb{F}_{q^n}$.
Let $L(x)=-x^{\sigma^d}+bx^\sigma+ax \in \mathcal{L}_{n,q,\sigma}$.
When $s=1$, Theorem \ref{th:mainMcG} is \cite[Theorem 1.1]{McGuireMuller}.
More precisely,

\begin{theorem} \label{McGuireMullers=1}
Let $d,n$ be positive integers such that $d\geq 3$.
\begin{enumerate}
    \item[(a)] If $n\leq (d-1)d$ and $d$ does not divide $n$, then there is no polynomial $L=x^{q^d}-bx^q-ax$ with $a,b \in \F_{q^n}$ that splits completely over $\F_{q^n}$.
    \item[(b)] Let $n=id$ with $i \in \{1,\ldots,d-1\}$. Let $L=x^{q^d}-bx^q-ax \in \F_{q^n}[x]$. Then $L$ has $q^d$ roots in $\F_{q^n}$ if and only if $a^{1+q^d+\cdots+q^{(i-1)d}}$ and $b=0$.
    \item[(c)] If $n=d(d-1)+1$, then $\dim_{\mathbb{F}_q}(\ker(ax+bx^q-x^{q^d}))=d$ with $a,b \in \mathbb{F}_{q^n}$ if and only if 
    \[ \left\{ \begin{array}{lll} \mathrm{N}_{q^n/q}(a)=(-1)^{d-1}, \\ b=-a^{q f_1}\,\,\text{where}\,\, f_1=\sum_{i=0}^{d-1}q^{id},\\ 
    d-1\,\, \text{is a power of}\,\, p.
    \end{array}\right. \]
    \end{enumerate}
\end{theorem}

Theorem \ref{th:mainMcG} is a consequence of Theorem \ref{McGuireMullers=1}, Theorem \ref{th:maxker} and of the following lemma.

\begin{lemma}\cite[Lemma 3.2]{LTZ}\label{lemma:LTZ}
Let $q$ be a prime power, and $\delta, n,s$ positive integers such that $\gcd(s,n)=1$ and $\delta <n$.
Let $U'$ be an $\F_{q^s}$subspace of $\mathbb{F}_{q^{sn}}$ with $\dim_{\F_{q^s}}(U')=\delta$ and let $U=U'\cap \F_{q^n}$.
Then $\dim_{\fq}(U)\leq \delta$.
\end{lemma}

\emph{Proof of Theorem \ref{th:mainMcG}}\\
Let $q'=q^s$, denote by $U'$ the kernel of $f(x)$ in $\mathbb{F}_{q'^n}$ and by $U$ the kernel of $f(x)$ in $\mathbb{F}_{q^n}$, that is $U=U'\cap \F_{q^n}$.
By Theorem \ref{Gow} $\dim_{\F_{q'}}(U')\leq d$ and Lemma \ref{lemma:LTZ} implies that if $\dim_{\F_{q}}(U)=d$ then $\dim_{\F_{q'}}(U')=d$.

\begin{itemize}
    \item[(a)] If $n\leq d(d-1)$ and $d\nmid n$, then $\dim_{\F_{q'}}(U')<d$ by (a) of Theorem \ref{McGuireMullers=1} and Lemma \ref{lemma:LTZ} implies $\dim_{\F_q}(U)<d$.
   \item[(b)] If $n\leq d(d-1)$, $d\mid n$ and $\dim_{\F_q}(U)=d$, then $\dim_{\F_{q'}}(U')=d$ and by (b) of Theorem \ref{McGuireMullers=1} $b= 0$ and $\N_{q'^{n}/q'}(a)=\N_{q^n/q}(a)= 1$, since $\gcd(s,n)=1$. The converse trivially holds, see e.g.\ \cite[Corollary 3.5]{teoremone}. 
   \item[(c)] Let $n=d(d-1)+1$  and assume that $f(x)$ has maximum kernel, that is $\dim_{\fq}(U)=d$ and hence $\dim_{\F_{q'}}(U')=d$. By (c) of Theorem \ref{McGuireMullers=1} we have
  \[ \left\{ \begin{array}{lll} \N_{q'^n/q'}(a)=\N_{q^n/q}(a)=(-1)^{d-1},\\
  b=-a^{q'\sum_{i=0}^{d-1}q'^{id}},\\
  d-1\,\,\text{is a power of}\,\,p.
  \end{array}\right.\]
  Conversely, suppose that \[ \left\{ \begin{array}{lll} 
  \N_{q^n/q}(a)=(-1)^{d-1},\\
  b=-a^{q'\sum_{i=0}^{d-1}q'^{id}},\\
  d-1\,\,\text{is a power of}\,\,p,
  \end{array}\right.\]
  then by (c) of Theorem \ref{McGuireMullers=1} $\dim_{\F_{q'}}(U')=d$.
  Since $f(x)$, seen as an $\F_{q'}$-linearized polynomial, has maximum kernel by Theorem \ref{th:maxker} 
  \[ A_n=C_LC_L^{q'}\cdot\ldots\cdot C_L^{q'^{n-1}}=I_d. \]
  As $C_L\in \F_{q^n}^{d\times d}$, $A_n$ coincides with $C_L C_L^{\sigma} \cdot \ldots \cdot C_L^{\sigma^{n-1}}$ and  from Theorem \ref{th:maxker} follows that the $\F_q$-linearized polynomial $f(x)$ has maximum kernel. \qed
\end{itemize}

\section{Proof of Theorem \ref{th:main}}\label{sec:thmain}

We divide two subsections: in the first we prove some preliminary results and we develop the machinery we use, and in the last subsection we show Theorem \ref{th:main}.

\subsection{Preliminaries results}

The $\sigma$-companion matrix $C_L$ of $L(x)=-x^{\sigma^d}+bx^\sigma+ax$ (defined as in Theorem \ref{th:maxker}) is
$$C_L=\left( \begin{matrix} 
0 & 0 & \cdots & 0 & a \\
1 & 0 & \cdots & 0 & b \\
0 & 1 & \cdots & 0 & 0 \\
\vdots & \vdots & \ddots & \vdots & \vdots \\
0 & 0 & \cdots & 1 & 0 
\end{matrix}
\right) \in \mathbb{F}_{q^n}^{d\times d}
.$$
We define $A_n=C_LC_L^\sigma \cdots C_L^{\sigma^{n-1}}$, where $C^\sigma$ denotes the matrix $C$ in which $\sigma$ is applied to each of its entries. By Theorem \ref{th:maxker}, $\dim_{\F_q}(\ker(L))=d$ if and only if $A_n=I_d$.

Let $M_{l,k}$ denote the $(l,d)$-entry of $A_k=C_LC_L^\sigma \cdots C_L^{\sigma^{k-1}}$, for $l\in \{1,\ldots,d\}$. 
The recursive relation of the next result will play a crucial role in the paper.

\begin{proposition}\label{prop:recursive}
Set $M_{l,l-d}=1$, and $M_{l,k}=0$ for $k \leq 0$ and $k \neq l-d$. 
\begin{itemize}
    \item For every $k \geq d$, the $(l,j)$-entry of $A_k$ is $M_{l,k-d+j}$.
    \item For every $1\leq l \leq d$ and $k \geq 1$, the following relation holds
\begin{equation} \label{recurvisel}
M_{l,k}=M_{l,k-d}a^{\sigma^{k-1}}+M_{l,k-d+1}b^{\sigma^{k-1}}.\end{equation}
\end{itemize}
\end{proposition}

\begin{proof}
First note that when multiplying a row $\mathbf{v}=(v_1,\ldots,v_d)$ by a Frobenius of the $\sigma$-companion matrix, namely $C^{\sigma^i}$, we get a row vector which may be obtained by shifting on the left the entries of $\mathbf{v}$ and by replacing  its $d$-th entry with the element $v_1a^{\sigma^i}+v_2b^{\sigma^i}$, that is
\begin{equation} 
 (v_1,\ldots,v_d) \left( \begin{matrix} 
0 & 0 & \cdots & 0 & a^{\sigma^i} \\
1 & 0 & \cdots & 0 & b^{\sigma^i} \\
0 & 1 & \cdots & 0 & 0 \\
\vdots & \vdots & \ddots & \vdots & \vdots \\
0 & 0 & \cdots & 1 & 0 
\end{matrix}
\right)=(v_2,\ldots,v_d,v_1a^{\sigma^i}+v_2b^{\sigma^i}).\end{equation}
Then we have:
$$A_1=A=
\left( \begin{matrix} 
0 & 0 & \cdots & 0 & M_{1,1} \\
1 & 0 & \cdots & 0 & M_{2,1} \\
\vdots & \vdots & & & \vdots \\
0 & 0 & \cdots & 1 & M_{d,1}
\end{matrix} \right), A_2=
\left( \begin{matrix} 
0 & 0 & \cdots & M_{1,1} & M_{1,2} \\
0 & 0 & \cdots & M_{2,1} & M_{2,2}\\
\vdots & \vdots & & & \vdots \\
0 & 0 & \cdots & M_{d,1} & M_{d,2} 
\end{matrix}\right), \ldots . $$

So, if $k \geq d$, we have that $M_{l,k-1}$ is the $(l,d-1)$-entry of $A_{k}$, $M_{l,k-2}$ is the $(l,d-2)$ entry of $A_{k}$ etc. In general, $M_{l,k-h}$ is the $(l,d-h)$-entry of $A_{k}$, for any $h\in\{0,1,\ldots,d-1\}$, that is $M_{l,k-d+j}$ is the $(l,j)$-entry of $A_k$, for $j=1,\ldots,d$, that proves the first part of the statement. 

Now, let $k \geq d$, we have
$$A_k= A_{k-1}C^{\sigma^{k-1}}=$$ $$= \left( \begin{matrix} 
M_{1,k-1-d+1} & M_{1,k-1-d+2} & \cdots  & M_{1,k-1} \\
M_{2,k-1-d+1} & M_{2,k-1-d+2}& \cdots & M_{2,k-1}\\
\vdots & \vdots & \vdots  & \vdots \\
M_{d,k-1-d+1} & M_{d,k-1-d+2} & \cdots & M_{d,k-1} 
\end{matrix}\right)
\left( \begin{matrix} 
0 & 0 & \cdots & 0 & a^{\sigma^{k-1}} \\
1 & 0 & \cdots & 0 & b^{\sigma^{k-1}} \\
0 & 1 & \cdots & 0 & 0 \\
\vdots & \vdots & \ddots & \vdots & \vdots \\
0 & 0 & \cdots & 1 & 0 
\end{matrix}
\right)
$$
and then 
$$M_{l,k}=M_{l,k-d}a^{\sigma^{k-1}}+M_{l,k-d+1}b^{\sigma^{k-1}},$$ that is Equation \eqref{recurvisel} when $k\geq d$.

Assume that $k\leq d-1$, then
$$M_{l,k}=(0,\cdots,0,1,0, \cdots,0, M_{l,1}, \cdots, M_{l,k-1}) \cdot \left(a^{\sigma^{k-1}}, b^{\sigma^{k-1}}, 0, \ldots, 0\right)^T$$
where the $1$ in the first vector is in the $(l-(k-1))$-position, 
so that $M_{l,k}$ is zero when $l-(k-1)\notin\{1,2\}$.  
If either $l-(k-1)=1$ or $l-(k-1)=2$, that is $k=l$ or $k=l-1$, we have $M_{l,l}=a^{\sigma^{l-1}}$ and
$M_{l,l-1}=b^{\sigma^{l-2}}$. 
Equation \eqref{recurvisel} then follows also for $k\leq d-1$.
\end{proof}

As a consequence of Proposition \ref{prop:recursive}, we get the following result. 

\begin{proposition} \label{propM1}
For any $i\in \{0,\ldots,d-3\}$ 
$M_{1,j}=0$,
for every $j\in \{id+2,\ldots,(i+1)d-(i+1)\}$.
\end{proposition}

\begin{proof}
We prove the claim by induction on $i$. For $i=0$ the assertion easily follows by \eqref{recurvisel}. Note that by Equation \eqref{recurvisel}, if $M_{1,k}=M_{1,k+1}=0$ then we have also $M_{1,k+d}=0$. So, if we assume the statement holds true for $i\in \{0,\ldots,d-4\}$, then we have $M_{1,id+2+d}=0,\ldots,M_{1,(i+1)d-(i+1)+d-1}=0$, that is the statement for $i+1$.
\end{proof}

For $M_{l,j}$'s the following holds.

\begin{proposition}\label{prop:Mlk}
For any $l\in \{2,\ldots,d\}$,
\[ M_{l,k}=\begin{cases}   
1 & \text{if}\,\, k=l-d,\\
0 & \text{if}\,\, k\leq 0 \,\,\text{and}\,\, k\ne l-d,\\
b^{\sigma^{l-2}} & \text{if}\,\, k=l-1,\\
a^{\sigma^{l-1}} & \text{if}\,\, k=l\\
0 & \text{if}\,\, 1 \leq k\leq d-1,\,\, k \neq l\,\, \text{and}\,\,k \neq l-1.
\end{cases} \]
\end{proposition}

\begin{proof}
The first part coincides with the definition of $M_{l,k}$ when $k \leq 0$. Now, let $k \geq 0$. 
As $k\leq d-1$, then by applying Equation \eqref{recurvisel}, $M_{l,k-d+1}\ne 0$ if and only if 
$k=l-1$.
Similarly, $M_{l,k-d}$ is nonzero if and only if $k=l$. The assertion then follows.
\end{proof}

Furthermore, according to \eqref{eq:Cfrelshort} of Theorem \ref{th:maxker}, we are interested in determining the $(l,1)$-entries of $A_n$. By Proposition \ref{prop:recursive}, $(l,1)$-entries of $A_n$ correspond to $M_{l,n-d+1}$ for $l\in\{1,2,\ldots,d\}$.  
So, Theorem \ref{th:maxker}
may be rephrased in our case as follows.

\begin{corollary} \label{systemsplit}
Let $n \geq 1$, $d\geq 3$, $a,b \in \F_{q^n}$ and $L(x)=-x^{\sigma^d}+bx^\sigma+ax \in \mathcal{L}_{n,q,\sigma}$.
Then $\dim_{\mathbb{F}_q}(\ker(L(x))=d$  if and only if 
$$M_{l,n-d+1}=\begin{cases} 1  & \mbox{if } l=1, \\
0 & \mbox{if } l=2,\ldots d. \end{cases}$$
\end{corollary}

Our aim now is to write $M_{l,k}$ as a combination of $M_{l,i}$'s which are easier to calculate.
This can be done by using the recursive formula \eqref{recurvisel}. Indeed, by applying \eqref{recurvisel}, after $j$ steps $M_{l,k}$ will depend on $M_{l,k-jd},\ldots,M_{l,k-jd+j}$, see Figure \ref{grafo}.

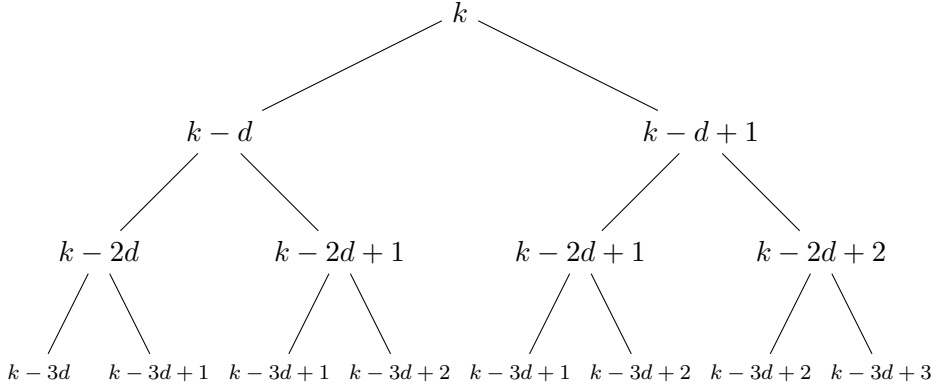
\begin{figure}[htp]	
\begin{tikzpicture}[scale=.8]
\draw
(1, 1) node(a){\scriptsize $k-3d$}
(3, 1) node(b){\scriptsize$k-3d+1$}
(5, 1) node(c){\scriptsize$k-3d+1$}
(7, 1) node(d){\scriptsize$k-3d+2$}
(9, 1) node(e){\scriptsize$k-3d+1$}
(11, 1) node(f){\scriptsize$k-3d+2$}
(13, 1) node(g){\scriptsize$k-3d+2$}
(15, 1) node(h){\scriptsize$k-3d+3$}
(2, 3) node(i){$k-2d$}
(6, 3) node(l){$k-2d+1$}
(10, 3) node(m){$k-2d+1$}
(14, 3) node(n){$k-2d+2$}
(4, 5) node(o){$k-d$}
(12, 5) node(p){$k-d+1$}
(8, 7) node(q){$k$};

\draw[-] (a) -- (i);
\draw[-] (i) -- (o);
\draw[-] (o) --  (q);
\draw[-] (q) -- (p);
\draw[-] (o) --  (l);
\draw[-] (i) -- (b);
\draw[-] (l) --  (c);
\draw[-] (l) -- (d);
\draw[-] (p) --  (m);
\draw[-] (p) -- (n);
\draw[-] (m) --  (e);
\draw[-] (m) -- (f);
\draw[-] (n) --  (g);
\draw[-] (n) -- (h);

\end{tikzpicture}
\caption{Recursive indexes of the $M_{i,j}$'s} \label{grafo}
\end{figure}

Let $j$ and $t$ be two positive integers with $0 \leq t \leq j$. The element $M_{l,k-jd+t}$ appears when we apply $j$ times the recursive formula \eqref{recurvisel} to $M_{l,k}$.
Taking into account Figure \ref{grafo}, we note that $M_{l,k-jd+t}$ is obtained $\binom{j}{t}$ times, each of which is obtained going right $t$ times and going left $j-t$ times in the Figure \ref{grafo}.

So, by applying $j$ times \eqref{recurvisel} we have
\begin{equation} \label{expressionMk}
M_{l,k}=\sum\limits_{t=0}^jc^k_{j,t}M_{l,k-jd+t},
\end{equation}
where $c_{j,t}^k \in \F_{q^n}$ denotes the coefficient of $M_{l,k-jd+t}$, which can be computed as follows.
As $c^k_{j,t}$ is the coefficient of $M_{l,k-jd+t}$, this means that in order to obtain $M_{l,k-jd+t}$ we go right through Figure \ref{grafo} exactly $t$ times at the steps $i_1,\ldots,i_t$ with $i_1\leq \ldots\leq i_t$.
Note that when we go right in Figure \ref{grafo} we multiply by a power of $b$, otherwise we multiply by a power of $a$.
More precisely, if we first go left $i_1$ times in the expression of $M_{l,k}$ will appear $M_{l,k-i_1d}$ multiplied by $a^{\sigma^{k-1}}\cdot \ldots \cdot a^{\sigma^{k-(i_1-1)d-1}}$. Then, if we go right in the expression of $M_{l,k}$ will appear $a^{\sigma^{k-1}}\cdot \ldots \cdot a^{\sigma^{k-(i_1-1)d-1}} M_{l,k-(i_1+1)d+1}$ multiplied by $b^{\sigma^{k-i_1d-1}}$. Therefore, we apply this procedure $j$ times and hence
\begin{equation}\label{eqckjt}
c^{k}_{j,t}=\sum_{\substack{i_0,\ldots,i_{t}=0 \\ i_0+\ldots+i_t=j-t}}^{j-t}  \left( \prod_{j_0=1}^{i_0} a^{\sigma^{k-(j_0-1)d-1}} \right) \left( \prod_{j_1=1}^{i_1} a^{\sigma^{k-(i_0+j_1)d}} \right)\cdot \ldots \cdot \end{equation}
$$\left( \prod_{j_{t-1}=1}^{i_{t-1}} a^{\sigma^{k-(i_0+\cdots+i_{t-2} +j_{t-1} +t-2)d+t-2}} \right) \cdot \left( \prod_{j_{t}=1}^{i_t} a^{\sigma^{k-(j-i_t+j_t-1)d+t-1}} \right) \cdot$$ 
$$ \left( \prod_{u=0}^{t-1} b^{\sigma^{k-(i_0+\cdots+i_u+u)d+u-1}} \right).
$$

Note that
\[ c_{j,0}^k=\prod_{j_0=1}^{j} a^{\sigma^{k-(j_0-1)d-1}}\,\,\,\, \text{and}\,\,\,\,  c_{j,j}^k=  \prod_{u=0}^{j-1} b^{\sigma^{k-u d+u-1}}. \]

\noindent In particular, for $k=n-d+1$, we have
$$M_{l,n-d+1}=\sum\limits_{t=0}^jc_{j,t}M_{l,n-(j+1)d+t+1},$$
where we set $c_{j,t}=c_{j,t}^{n-d+1}$.

Let $n=d(d-1)+g$ with $0\leq g< d$. We have 
$$M_{l,n-d+1}=\sum\limits_{t=0}^jc_{j,t}M_{l,(d-j-2)d+t+1+g},$$
and choosing $j=d-1$, we have
\begin{equation}\label{eq:importante}
M_{l,n-d+1}=\sum_{t=0}^{d-1}c_{d-1,t}M_{l,-d+t+1+g}.
\end{equation}

The first case $g=1$ and $\sigma \colon x\in \F_{q^n}\mapsto x^q\in \F_{q^n}$ is studied in \cite{McGuireMuller}.

When $a$ and $b$ satisfies certain conditions, the $c_{i,j}^k$ have a simpler representation. 
 
\begin{lemma} \label{trianglepascal}
Suppose that $a^{\sigma^{d}}b=a^{\sigma}b^{\sigma^d}$, then we have 
$$M_{l,k}=\sum_{i=0}^{j} \binom{j}{i}z_{j,i}^kM_{l,k-jd+i},$$
for every positive integer $k\geq d+1$ and $j\in \{0,\ldots,\lfloor \frac{k-1}{d}+1 \rfloor\}$.
Hence, the $c_{j,i}^k$ defined in \eqref{expressionMk} are 
$$c_{j,i}^k=\binom{j}{i}z_{j,i}^k,$$
and the $z^k_{j,i}$ are determined by the following recursion:
\begin{equation} \label{expressionzij}
z_{j,i}^k=\begin{cases}
1 & \mbox{if }j=i=0, \\
z_{j-1,0}^ka^{\sigma^{k-(j-1)d-1}} & \mbox{if }i=0,  \\
z_{j-1,i}^ka^{\sigma^{k-(j-1)d+i-1}}=z_{j-1,i-1}^kb^{\sigma^{k-(j-1)d+i-2}} & \mbox{if }0<i<j, \\
z_{j-1,j-1}^kb^{\sigma^{k-(j-1)d+j-2}} & \mbox{if }i=j.
\end{cases} 
\end{equation}
\end{lemma}

\begin{proof}
First observe that when $k\geq d+1$ then $a^{\sigma^{k-1}}b^{\sigma^{k-d-1}}=a^{\sigma^{k-d}}b^{\sigma^{k-1}}$.
We will prove the statement by induction.
The case $j=1$ follows by \eqref{recurvisel}; now we prove the assertion for $j=2$.
Let $k \geq d+1$. By Equation \eqref{recurvisel},
\begin{equation} \label{recurs2}
    \begin{split}
        M_{l,k} & =M_{l,k-d}a^{\sigma^{k-1}}+M_{l,k-d+1}b^{\sigma^{k-1}} \\
        & =(M_{l,k-2d}a^{\sigma^{k-d-1}}+M_{l,k-2d+1}b^{\sigma^{k-d-1}})a^{\sigma^{k-1}}+ \\
        & \hskip 0.5 cm (M_{l,k-2d+1}a^{\sigma^{k-d}}+M_{l,k-2d+2}b^{\sigma^{k-d}})b^{\sigma^{k-1}} \\ 
        & =M_{l,k-2d}a^{\sigma^{k-d-1}+\sigma^{k-1}}+M_{l,k-2d+1}(a^{\sigma^{k-1}}b^{\sigma^{k-d-1}}+a^{\sigma^{k-d}}b^{\sigma^{k-1}}) \\
        & \hskip 0.5 cm + M_{l,k-2d+2}b^{\sigma^{k-d}+\sigma^{k-1}} 
    \end{split}
\end{equation}
As $a^{\sigma^{k-1}}b^{\sigma^{k-d-1}}=a^{\sigma^{k-d}}b^{\sigma^{k-1}}$, putting $z^k_{2,0}=a^{\sigma^{k-d-1}+\sigma^{k-1}}$, \\  $z^k_{2,1}=2a^{\sigma^{k-d}}b^{\sigma^{k-1}}$ and $z^k_{2,2}=b^{\sigma^{k-d}+\sigma^{k-1}}$, \eqref{recurs2} may be rewritten as follows 
$$M_{l,k}=z^k_{2,0}M_{l,k-2d}+2z^k_{2,1}M_{l,k-2d+1}+z^k_{2,2}M_{l,k-2d+2},$$
and the statement for $j=2$ is proved.
Now, assume that the statement holds for any $t<j$, that is 
$$M_{l,k}=\sum_{i=0}^{t} \binom{t}{i}z_{t,i}^kM_{l,k-td+i},$$
where $z_{t,i}^k$ are defined as in \eqref{expressionzij}.
Then by applying \eqref{recurvisel} to $M_{l,k-(j-1)d+i}$ we have
\begin{small}
\[
\begin{split}
    M_{l,k} & =\sum_{i=0}^{j-1}\binom{j-1}{i}z^k_{j-1,i}M_{l,k-(j-1)d+i}\\
    & =\sum_{i=0}^{j-1}\binom{j-1}{i}z^k_{j-1,i}(M_{l,k-jd+i}a^{\sigma^{k-(j-1)d+i-1}}+M_{l,k-jd+i+1}b^{\sigma^{k-(j-1)d+i-1}}) \\
    &=\binom{j-1}{0}z^k_{j-1,0}M_{l,k-jd}a^{\sigma^{k-(j-1)d-1}}+\binom{j-1}{j-1}z^k_{j-1,j-1} M_{l,k-jd+j}b^{\sigma^{k-(j-1)d+j-2}} \\
    & +\sum_{i=1}^{j-1}M_{l,k-jd+i} \left(\binom{j-1}{i}z^k_{j-1,i}a^{\sigma^{k-(j-1)d+i-1}}+ \binom{j-1}{i-1}z^k_{j-1,i-1}b^{\sigma^{k-(j-1)d+i-2}}\right).
    \end{split}
\]
\end{small}
By hypothesis we have
\begin{equation} \label{tartrela}
a^{\sigma^{k-(j-2)d+i-2}}b^{\sigma^{k-(j-1)d+i-2}}=a^{\sigma^{k-(j-1)d+i-1}}b^{\sigma^{k-(j-2)d+i-2}}.
\end{equation}
Furthermore,
$$z_{j-1,i}^ka^{\sigma^{k-(j-1)d+i-1}}=z_{j-2,i-1}^ka^{\sigma^{k-(j-1)d+i-1}}b^{\sigma^{k-(j-2)d+i-2}}$$
and
$$z_{j-1,i-1}^k b^{\sigma^{k-(j-1)d+i-2}}=z_{j-2,i-1}^ka^{\sigma^{k-(j-2)d+i-2}}b^{\sigma^{k-(j-1)d+i-2}}$$
and by \eqref{tartrela}, these two expressions are equal. 
As $\binom{j-1}{i}+ \binom{j-1}{i-1}=\binom{j}{i}$ we get
$$M_{l,k}=\binom{j}{0}z^k_{j-1,0}M_{l,k-jd}a^{\sigma^{k-(j-1)d-1}}+\binom{j}{j}z^k_{j-1,j-1} M_{l,k-jd+j}b^{\sigma^{k-(j-1)d+j-2}}$$ $$
     +\sum_{i=1}^{j-1} \binom{j}{i}z^k_{j-1,i}M_{l,k-jd+i}a^{\sigma^{k-(j-1)d+i-1}},
$$ 
that is the statement.
\end{proof}

\subsection{Proof of Theorem \ref{th:main}}\label{sec:closedfor}

Let $n=(d-1)d+g$ and $1\leq g \leq d-1$.
We also remind that in Proposition \ref{prop:recursive}, we set $M_{l,l-d}=1$, and $M_{l,k}=0$ for $k \leq 0$ and $k \neq l-d$. \\

\emph{Proof of Theorem \ref{th:main}}\\
Let start by recalling that, by \eqref{eq:importante}, we have
\begin{equation}\label{eq:Ml,n-d+1}
M_{l,n-d+1}=\sum_{t=0}^{d-1} c_{d-1,t} M_{l,-d+1+g+t}.
\end{equation}

\noindent By Theorem \ref{th:maxker}, $\dim_{\F_q}(\ker(L(x)))=d$ if and only if 
$M_{l,n-d+1}=
\begin{cases}
1 & \text{if}\,\, l=1,\\
0 & \text{if}\,\, l\ne 1.
\end{cases}$
So, \eqref{eq:Ml,n-d+1} and Proposition \ref{prop:Mlk} imply that if $g \leq d-2$
$$\begin{cases}
1=M_{1,n-d+1}=M_{1,1}c_{d-1,d-g}=ac_{d-1,d-g} \\
0=M_{2,n-d+1}=M_{2,1}c_{d-1,d-g}+M_{2,2}c_{d-1,d-g+1}\\
\hskip 2.38 cm =bc_{d-1,d-g}+a^\sigma c_{d-1,d-g+1}\\
\vdots \\
0=M_{g,n-d+1}=c_{d-1,d-2}M_{g,g-1}+c_{d-1,d-1}M_{g,g}\\
\hskip 2.38 cm = b^{\sigma^{g-2}}c_{d-1,d-2}+a^{\sigma^{g-1}}c_{d-1,d-1}\\
0=M_{g+1,n-d+1}=c_{d-1,d-1}M_{g+1,g}+c_{d-1,0}M_{g+1,-d+g+1}\\
\hskip 2.77 cm =b^{\sigma^{g-1}}c_{d-1,d-1}+c_{d-1,0} \\
0=M_{g+2,n-d+1}=c_{d-1,1}M_{g+2,-d+1+g+1}=c_{d-1,1} \\
\vdots \\
0=M_{d,n-d+1}=c_{d-1,d-g-1}
\end{cases}
$$
and if $g=d-1$
$$\begin{cases}
1=M_{1,n-d+1}=M_{1,1}c_{d-1,d-g}=ac_{d-1,1} \\
0=M_{2,n-d+1}=M_{2,1}c_{d-1,1}+M_{2,2}c_{d-1,2}=bc_{d-1,1}+a^\sigma c_{d-1,2}\\
\vdots \\
0=M_{d-1,n-d+1}=c_{d-1,d-2}M_{d-1,d-2}+c_{d-1,d-1}M_{d-1,d-1}\\
\hskip 2.8 cm = b^{\sigma^{d-3}}c_{d-1,d-2}+a^{\sigma^{d-2}}c_{d-1,d-1}\\
0=M_{d,n-d+1}=c_{d-1,d-1}M_{d,d-1}+c_{d-1,0}M_{d,0}\\
\hskip 2.4 cm =b^{\sigma^{d-2}}c_{d-1,d-1}+c_{d-1,0},
\end{cases}
$$

which can be written as 

$$\begin{cases}
c_{d-1,d-g}=a^{-1} \\
c_{d-1,d-g+1}=-ba^{-1}a^{-\sigma}\\
\vdots \\
c_{d-1,d-g+u}=(-1)^ub^{\sigma^{u-1}}b^{\sigma^{u-2}}\cdots b^\sigma b a^{-1}a^{-\sigma}\cdots a ^{-\sigma^{u}} \\
\vdots \\
c_{d-1,d-1}=(-1)^{g-1}b^{\sigma^{g-2}}b^{\sigma^{g-3}}\cdots b^\sigma b a^{-1}a^{-\sigma}\cdots a ^{-\sigma^{g-1}}\\
\hskip 1.45 cm =(-1)^{g-1}b^{\frac{\sigma^{g-1}-1}{\sigma-1}}a^{-\frac{\sigma^g-1}{\sigma-1}} \\
c_{d-1,0}=(-1)^{g}b^{\sigma^{g-1}}b^{\sigma^{g-2}}\cdots b^\sigma b a^{-1}a^{-\sigma}\cdots a ^{-\sigma^{g-1}}\\
\hskip 1.05 cm =(-1)^{g}b^{\frac{\sigma^g-1}{\sigma-1}}a^{-\frac{\sigma^g-1}{\sigma-1}} \\
c_{d-1,1}=0 \\
\vdots \\
c_{d-1,d-g-1}=0
\end{cases}
$$
if $g \leq d-2$ and

\begin{equation} \label{trinomcondition} \begin{cases}
c_{d-1,1}=a^{-1} \\
c_{d-1,2}=-ba^{-1}a^{-\sigma}\\
\vdots \\
c_{d-1,1+u}=(-1)^u \prod_{i=0}^{u-1}b^{\sigma^{i}}\prod_{i=0}^{u}a^{\sigma^{i}} \\
\vdots \\
c_{d-1,d-1}=(-1)^{d-2}b^{\sigma^{d-3}}b^{\sigma^{d-4}}\cdots b^\sigma b a^{-1}a^{-\sigma}\cdots a ^{-\sigma^{d-2}}\\
\hskip 1.45 cm =(-1)^{d-2}b^{\frac{\sigma^{d-2}-1}{\sigma-1}}a^{-\frac{\sigma^{d-1}-1}{\sigma-1}} \\
c_{d-1,0}=(-1)^{d-1}b^{\sigma^{d-2}}b^{\sigma^{d-3}}\cdots b^\sigma b a^{-1}a^{-\sigma}\cdots a ^{-\sigma^{d-2}}\\
\hskip 1.04 cm =(-1)^{d-1} b^{\frac{\sigma^{d-1}-1}{\sigma-1}}a^{-\frac{\sigma^{d-1}-1}{\sigma-1}} \\
\end{cases}
\end{equation}
if $g =d-1$.

In order to determine the $c_{d-1,i}$'s we use Equation \eqref{eqckjt} with $j=d-1$ and $k=n-d+1$ and we get

$$
c^{k}_{j,t}=\sum_{\substack{i_0,\ldots,i_{t}=0 \\ i_0+\ldots+i_t=j-t}}^{j-t}  \left( \prod_{j_0=1}^{i_0} a^{\sigma^{k-(j_0-1)d-1}} \right) \left( \prod_{j_1=1}^{i_1} a^{\sigma^{k-(i_0+j_1)d}} \right)\cdot \ldots \cdot $$
$$\left( \prod_{j_{t-1}=1}^{i_{t-1}} a^{\sigma^{k-(i_0+\cdots+i_{t-2} +j_{t-1} +t-2)d+t-2}} \right) \cdot \left( \prod_{j_{t}=1}^{i_t} a^{\sigma^{k-(j-i_t+j_t-1)d+t-1}} \right)\cdot $$ $$ \left( \prod_{u=0}^{t-1} b^{\sigma^{k-(i_0+\cdots+i_u+u)d+u-1}} \right).
$$

Hence, by looking to the previous systems and the above relation we have

$$
c_{d-1,t}=\sum_{\substack{i_0,\ldots,i_{t}=0 \\ i_0+\ldots+i_t=d-1-t}}^{d-1-t}  \left( \prod_{j_0=1}^{i_0} a^{\sigma^{d(d-1-j_0)+g}} \right) \left( \prod_{j_1=1}^{i_1} a^{\sigma^{d(d-i_0-j_1-2)+g+1}} \right)\cdot \ldots \cdot $$
$$\left( \prod_{j_{t-1}=1}^{i_{t-1}} a^{\sigma^{d(d-i_0+\cdots-i_{t-2} -j_{t-1} -t)+g+t-1}} \right) \cdot \left( \prod_{j_{t}=1}^{i_t} a^{\sigma^{d(i_t-j_t)+t+g}} \right) \cdot$$ $$\left( \prod_{u=0}^{t-1} b^{\sigma^{d(d-2-i_0-\cdots-i_u-u)+u+g}} \right)$$ 
$$=\left\{ \begin{array}{lll} (-1)^v b^{\frac{\sigma^{v}-1}{\sigma-1}} a^{-\frac{\sigma^{v+1}-1}{\sigma-1}} & \text{if}\,\, t=d-g+v, \,\, 0\leq v\leq 
g-1,\\
(-1)^g a^{-\frac{\sigma^{g}-1}{\sigma-1}}b^{\frac{\sigma^{g}-1}{\sigma-1}}& \text{if}\,\, t=0,\\ 
0 & \text{otherwise},\end{array}\right. 
$$

which proves the assertion. \qed

\begin{remark}
We point out that the proof of (c) of Theorem \ref{th:mainMcG} can be also obtained by manipulating the conditions of Theorem \ref{th:main}.
\end{remark}

\section{Some necessary conditions}\label{sec:neccond}

In this section we will give some necessary condition on the coefficients of $L(x)=ax+bx^{\sigma}-x^{\sigma^d}\in \mathcal{L}_{n,q,\sigma}$ to have kernel of maximum dimension.
In order to do this, we find the coefficients $c_{d-1,0}$ and $c_{d-1,d-1}$.
Indeed, the coefficient $c_{d-1,0}^k$ can be obtained by applying \eqref{eqckjt} with $t=0$ and $j=d-1$, that is
\[ c_{d-1,0}^k=a^{\sigma^{k-1}+\sigma^{k-d-1}+\ldots+\sigma^{k-(d-2)d-1}}. \]
With $k=n-d+1$, we get
\begin{equation} \label{cd-10}
     c_{d-1,0}=c_{d-1,0}^k=a^{\sigma^g(\sigma^{d(d-2)}+\ldots+\sigma^d+1)}=a^{\sigma^g e_1}, 
\end{equation}
where $e_1=\frac{\sigma^{d(d-1)}-1}{\sigma^d-1}$.
Similarly, the coefficient $c_{d-1,d-1}^k$ can be obtained by applying \eqref{eqckjt} with $t=d-1$ and $j=d-1$, that is
\[ c_{d-1,d-1}^k=b^{\sigma^{k-1}+\sigma^{k-d}+\ldots+\sigma^{k-(d-2)d+d-3}}. \]
With $k=n-d+1$, we get
\begin{equation} \label{cd-1d-1}
 c_{d-1,d-1}=c_{d-1,d-1}^{n-d+1}=b^{\sigma^{g-1}(\sigma^{d-1}+\ldots+\sigma^{(d-1)(d-1)})}=b^{\sigma^{g-1} e_2}, 
 \end{equation}
where $e_2=\frac{\sigma^{d(d-1)}-\sigma^{d-1}}{\sigma^{d-1}-1}$.

As a consequence of, we get the following necessary conditions.

\begin{corollary} \label{cor:neccond}
Let $n=(d-1)d+g$, with $1\leq g \leq d-1$.
If $L(x)$ has maximum kernel then
$$ \begin{cases}
b^{\frac{\sigma^g-1}{\sigma-1}}a^{-\frac{\sigma^g-1}{\sigma-1}}=(-1)^{g}a^{\sigma^ge_1}, \\
b^{\frac{\sigma^{g-1}-1}{\sigma-1}}a^{-\frac{\sigma^g-1}{\sigma-1}}=(-1)^{g-1}b^{\sigma^{g-1}e_2}.
\end{cases}
$$
\end{corollary}

\section{A new family of $\sigma$-linearized trinomials with maximum kernel}\label{sec:newfamily}

In this section we introduce a new family of $\sigma$-linearized trinomials with maximum kernel, relying on the results of Section \ref{sec:thmain}.
In particular we prove the following.

\begin{theorem} \label{zullopaultheo}
Let $L(x)=-x^{\sigma^d}+bx^{\sigma}+ax \in \mathcal{L}_{n,q,\sigma}$, with $n=(d-1)d+d-1$. If
\begin{itemize}
\item $q$ is even,
\item $d$ is a power of $2$,
\item $\N_{q^n/q}(a)=1$,
\item $b=a^{-\sigma^d\frac{\sigma^{d(d-1)}-1}{\sigma^d-1}}$,
\end{itemize}
then $\dim_{\F_q}(\ker(L))=d$.
\end{theorem}

\begin{lemma}\label{finalteorecurs}
If $a,b \in \mathbb{F}_{q^n}$ with $b=a^{-\sigma^d\frac{\sigma^{d(d-1)}-1}{\sigma^d-1}}$ then
$$a^{\sigma^{d}}b=a^{\sigma}b^{\sigma^{d}},$$
and hence
$$a^{\sigma^{k-1}}b^{\sigma^{k-d-1}}=a^{\sigma^{k-d}}b^{\sigma^{k-1}},$$
for every positive integer $k\geq d+1$.
Moreover, for every $r\in\{1,\ldots, d-1\}$ and $u\leq r-1$,
\[b^{\sigma^{d(r-u)+u-1}}a^{-\sum_{i=1}^{r-u} \sigma^{id+(u-1)}}=a^{-\sigma^u}b^{-\sigma^{u-1}}a^{-\sum_{i=1}^{r-u-1}\sigma^{id+u} }.\]
\end{lemma}
\begin{proof}
We have
\[
\begin{split}
a^{\sigma^{d}}b & =a^{\sigma^{d}-\sum_{i=1}^{d-1}\sigma^{id}} \\
& =a^{-\sum_{i=2}^{d-1}\sigma^{id}} \\
& =a^{-\sum_{i=2}^{d}\sigma^{id}+\sigma^{d^2} } \\
 & =a^{-\sum_{i=2}^{d}\sigma^{id}+\sigma } \pmod{a^{\sigma^n}-a} \\
& =a^{\sigma}a^{-\sigma^{d}\sum_{i=2}^{d}\sigma^{(i-1)d} }  \\
 & =a^{\sigma}b^{\sigma^{d}}. \\
 \end{split}
\]
For the second part, using $b=a^{-\sigma^d\frac{\sigma^{d(d-1)}-1}{\sigma^d-1}}$, $\sigma^n=\mathrm{id}$ and $n=d^2-1$, we have
\[ 
\begin{split}
b^{\sigma^{d(r-u)+u-1}}& =a^{-\sum_{i=1}^{d-1}\sigma^{d(r-u+i)+u-1}} \\
&=a^{-\sigma^{u-1}(\sum_{i=1}^{u-r+d-1}\sigma^{d(r-u+i)}     +      \sum_{i=u-r+d}^{d-1}\sigma^{d(r-u+i)}        ) } \\
&=a^{-\sigma^{u-1}(\sum_{i=1}^{u-r+d-1}\sigma^{d(r-u+i)})}a^{-\sigma^{u-1}(\sum_{i=u-r+d}^{d-1}\sigma^{d(r-u+i)}) },
\end{split}
\]
and note that $d+u-r-1 \leq d-2$
and $u+d-r-1 \geq u \geq 1$.\\

Putting together
$a^{-\sum_{i=1}^{r-u} \sigma^{id+(u-1)}}$ and  $a^{-\sigma^{u-1}(\sum_{i=1}^{u-r+d-1}\sigma^{d(r-u+i)})}$
we obtain $b^{\sigma^{u-1}}$.

Now consider $a^{-\sigma^{u-1}(\sum_{i=u-r+d}^{d-1}\sigma^{d(r-u+i)}) }$.
Then it is equal to
$$a^{-\sigma^{u-1}(\sum_{i=d}^{d-1+r-u}\sigma^{id}) }=a^{-\sigma^{u-1}(\sum_{i=0}^{r-u-1}\sigma^{id+1}) }=a^{-\sigma^u}a^{-\sum_{i=1}^{r-u-1}\sigma^{id+u} }.$$
So, the assertion has been proved.
\end{proof}

The first part of the previous lemma allows us to apply Lemma \ref{trianglepascal}.

\begin{remark} \label{expressionznd2}
With $k=n-d+1$ and $n=d(d-1)+d-1=d^2-1$, $z_{j,i}^k$ in \eqref{expressionzij}, setting $z_{j,i}=z_{j,i}^{n-d+1}$, becomes:
\begin{equation}
z_{j,i}=z_{j,i}^{n-d+1}=\begin{cases}
1 & \mbox{if }j=i=0, \\
z_{j-1,0}a^{\sigma^{d(d-j)-1}} & \mbox{for }i=0,  \\
z_{j-1,i}a^{\sigma^{d(d-j)+i-1}}=z_{j-1,i-1}b^{\sigma^{d(d-j)+i-2}} & \mbox{for }0<i<j, \\
z_{j-1,j-1}b^{\sigma^{d(d-j)+j-2}} & \mbox{for }i=j.
\end{cases} 
\end{equation}
Furthermore, by Lemma \ref{trianglepascal} we have $$c_{j,i}=\binom{j}{i}z_{j,i}.$$
\end{remark}

In the next lemma we compute the $c_{d-1,i}$'s.

\begin{lemma}\label{lemma:ci,jfin}
For any $r\in\{2,\ldots, d-1\}$, $$c_{d-1,r}=\binom{d-1}{r}\prod_{v=1}^{r-1} b^{\sigma^{(r-v)d+v-1}} a^{-\sum_{i=1}^{r-1}\sigma^{id}}a^{-1}.$$
and
\begin{equation} 
    c_{d-1,1}=\binom{d-1}{1}a^{-1} .
\end{equation}
\end{lemma}
\begin{proof}
By Remark \ref{expressionznd2}, for any $r\geq 2$

\[ \begin{split}
c_{d-1,r}& = \binom{d-1}{r}b^{\sigma^{d+r-2}}z_{d-2,r-1} \\
&=\binom{d-1}{r}b^{\sigma^{d+r-2}}b^{\sigma^{2d+r-3}}z_{d-3,r-2} \\
&=\binom{d-1}{r}\prod_{v=0}^{r-1} b^{\sigma^{(r-v)d+v-1}}z_{d-(r+1),0}\\
\end{split}
\]
and then applying Remark \ref{expressionznd2} to the case $z_{j,0}$ it follows
\[ \begin{split}
c_{d-1,r}
&=\binom{d-1}{r}\prod_{v=0}^{r-1} b^{\sigma^{(r-v)d+v-1}} a^{\sigma^{d(d-(d-(r+1)))-1}+\cdots+\sigma^{d(d-1)-1}}   \\
&=\binom{d-1}{r}\prod_{v=0}^{r-1} b^{\sigma^{(r-v)d+v-1}}a^{\sigma^{d(r+1)-1}+\cdots+\sigma^{d(d-1)-1}}      \\
\end{split}
\]

From $v=0$, in the product appears
$b^{\sigma^{rd-1}}a^{\sigma^{d(r+1)-1}+\cdots+\sigma^{d(d-1)-1}}$ which, using $b=a^{-\sigma^d\frac{\sigma^{d(d-1)}-1}{\sigma^d-1}}$, equals to
$$ a^{-\sigma^{d^2-1}-\sigma^{d(d+1)-1}-\cdots-\sigma^{d(d-1+r)-1}}=a^{-1}a^{-\sigma^{d}-\cdots-\sigma^{d(-1+r)}}, $$
so that
$$c_{d-1,r}=\binom{d-1}{r}\prod_{v=1}^{r-1} b^{\sigma^{(r-v)d+v-1}} a^{-\sum_{i=1}^{r-1}\sigma^{id}}a^{-1}.$$
When $r=1$
\begin{equation} \label{calcolocd-11}
\begin{split}
c_{d-1,1}& =\binom{d-1}{1}b^{\sigma^{d-1}}z_{d-2,0} \\
& =\binom{d-1}{1}b^{\sigma^{d-1}}a^{\sigma^{2d-1}}a^{\sigma^{3d-1}}\cdots a^{\sigma^{d(d-1)-1}}\\
& =\binom{d-1}{1}a^{-\sum_{i=1}^{d-1}\sigma^{(i+1)d-1}+\sigma^{2d-1}+\sigma^{3d-1}+\cdots +\sigma^{d(d-1)-1}} \\
& =\binom{d-1}{1}a^{-\sigma^{d^2-1}}\\
&=\binom{d-1}{1}a^{-1}. 
\end{split}
\end{equation}
\end{proof}

Now we are ready to prove Theorem \ref{zullopaultheo}.\\

\emph{Proof of Theorem \ref{zullopaultheo}}\\
By Lemma \ref{lemma:ci,jfin}, for any $r\geq 2$ it follows
$$c_{d-1,r}=\binom{d-1}{r}\prod_{v=1}^{r-1} b^{\sigma^{(r-v)d+v-1}} a^{-\sum_{i=1}^{r-1}\sigma^{id}}a^{-1}.$$
Applying $r-1$ times Lemma \ref{finalteorecurs} to $c_{d-1,r}$
we have that 
\begin{equation}\label{calcolocd-1k}
\begin{split}
c_{d-1,r}&=\binom{d-1}{r} b^{\sigma^{(r-1)d}}
a^{-\sum_{i=1}^{r-1}\sigma^{id}}\prod_{v=2}^{r-1} b^{\sigma^{(r-v)d+v-1}}a^{-1}
\\&=\binom{d-1}{r} \prod_{v=2}^{r-1} b^{\sigma^{(r-v)d+v-1}} a^{-\sum_{i=1}^{r-2}\sigma^{id+1}} b a^{-\sigma -1}
\\&=\vdots
\\&=\binom{d-1}{r}\prod_{i=0}^{r-2}b^{\sigma^{i}} \prod_{i=0}^{r-1}a^{-\sigma^{i}},
\end{split}
\end{equation}
for $r \geq 2$. 

Recall that $L(x)$ has maximum kernel if and only if $a$ and $b$ satisfy System \eqref{trinomcondition}. 
So, substituting \eqref{calcolocd-11} and \eqref{calcolocd-1k} in \eqref{trinomcondition} we obtain

\begin{equation} \label{verifysystem}
\begin{cases}
\binom{d-1}{1}a^{-1}=a^{-1} \\
\vdots \\
\binom{d-1}{r} \prod_{i=0}^{r-2}b^{\sigma^{i}} \prod_{i=0}^{r-1}a^{-\sigma^{i}}=\prod_{i=0}^{r-2}b^{\sigma^{i}} \prod_{i=0}^{r-1}a^{-\sigma^{i}} \\
\vdots \\
c_{d-1,0}= b^{\frac{\sigma^{d-1}-1}{\sigma-1}}a^{-\frac{\sigma^{d-1}-1}{\sigma-1}} \\
\end{cases}
\end{equation}

From \eqref{cd-10} we have
$$c_{d-1,0}=a^{\frac{\sigma^{d^2-1}-\sigma^{d-1}}{\sigma^d-1}},$$
so that the last equation of \eqref{verifysystem}, since $\sigma^{d^2}=\sigma$, becomes
$$a^{\frac{\sigma^{d^2-1}-\sigma^{d-1}}{\sigma^d-1}}=b^{\frac{\sigma^{d-1}-1}{\sigma-1}}a^{-\frac{\sigma^{d-1}-1}{\sigma-1}},$$
that is
$$a^{\frac{\sigma^{d^2-1}-\sigma^{d-1}}{\sigma^d-1}}=a^{-\frac{\sigma^{d^2}-\sigma^d}{\sigma^d-1}\frac{\sigma^{d-1}-1}{\sigma-1}}a^{-\frac{\sigma^{d-1}-1}{\sigma-1}},$$
which is equivalent to $\N_{q^n/q}(a)=1$ and so the last equation of System \eqref{verifysystem} is satisfied.

By Lucas's Theorem we have
\[ \binom{d-1}{r}=\prod_{v=0}^u \binom{d_v}{r_v} \pmod{2}, \]
where $d-1=d_u 2^u+\ldots+ d_1 2+ d_0$ and $r=r_u 2^u+\ldots+ r_1 2+ r_0$.
As $d$ is a power of $2$, namely $d=2^{u+1}$, then
\[ d-1=2^{u}+\ldots+2+1, \]
so that $d_v=1$ for each $v \in \{0,\ldots,u\}$ and Lucas's Theorem implies that $\binom{d-1}{r}=1$.
Therefore, the remaining equations of System \eqref{verifysystem} are satisfied. \qed

\section{$\sigma$-degree $3$ and $4$}\label{sec:d=3,4}

In this section we will make use of Theorems \ref{th:mainMcG} and \ref{th:main} in order to to give explicit relations on $a,b \in \F_{q^n}$ ensuring that the polynomial $L(x)=ax+bx^{\sigma}-x^{\sigma^d}$ has maximum kernel when $d\in \{3,4\}$.

\begin{proposition}\label{prop:d=3}
Let $L(x)=ax+bx^{\sigma}-x^{\sigma^3}\in \mathcal{L}_{n,q,\sigma}$ with $\sigma$ a generator of $\mathrm{Gal}(\F_{q^n}\colon\F_q)$.
\begin{itemize}
    \item If $n\leq 5$ and $n \neq 3$ then $\dim_{\F_q}(\ker(L(x)))<3$;
    \item if $n\in \{3,6\}$ then $\dim_{\F_q}(\ker(L(x)))=3$ if and only if $b=0$ and $\N_{q^6/q^3}(a)=1$;
    \item if $n=7$ then $\dim_{\F_q}(\ker(L(x)))=3$ if and only if
    \[\begin{cases} 
    \N_{q^7/q}(a)=1,\\
    b=-a^{\sigma(1+\sigma^3+\sigma^6)},\\
    q\,\,\text{is even};
    \end{cases}\]
    \item if $n=8$ then $\dim_{\F_q}(\ker(L(x)))=3$ if and only if
    \[ \begin{cases}
    \N_{q^8/q}(a)=1,\\
    q \equiv 0,2 \pmod{3},\\
    b=\frac{-\alpha}{a^{\sigma^6+\sigma^3}},\\
    \alpha\,\,\text{is a root of}\,\, x^2+x+1.
    \end{cases}\]
\end{itemize}
\end{proposition}

\begin{proof}
By Theorem \ref{th:mainMcG}, the only part that remains to prove is the last point.
Let $n=8$. 
By Theorem \ref{th:main}, $\dim_{\F_q}(\ker(L(x)))=3$ if and only if 
\begin{equation}\label{eq:sist1}\begin{cases}
b^{\sigma^{5}}a^{\sigma^{3}}+a^{\sigma^{5}}b^{\sigma^2}=a^{-1}, \\
b^{\sigma^5+\sigma^3}=-ba^{-1-\sigma},\\
a^{\sigma^5+\sigma^2}=b^{\sigma+1}a^{-1-\sigma},
\end{cases}
\end{equation}
from which clearly we have $a,b \ne 0$.
System \eqref{eq:sist1} is equivalent to the following system
\begin{equation}\label{eq:sist2}\begin{cases}
b^{\sigma^{5}}a^{\sigma^{3}+1}+a^{\sigma^{5}+1}b^{\sigma^2}=1, \\
b^{\sigma^5+\sigma^3+\sigma}=-a^{\sigma^5+\sigma^2},\\
a^{\sigma^5+\sigma^2+\sigma+1}=b^{\sigma+1}.
\end{cases}
\end{equation}
Clearly, by the second equation of the above system we also have $b^{\sigma^6+\sigma^4+\sigma^2}=-a^{\sigma^6+\sigma^3}$ and $b^{\sigma^7+\sigma^5+\sigma^3}=-a^{\sigma^7+\sigma^4}$, and hence by the third equation of System \eqref{eq:sist2} the equation $\N_{q^8/q}(a)=\N_{q^8/q}(b)$. Therefore, System \eqref{eq:sist2} is equivalent to 
$$\begin{cases}
b^{\sigma^{5}}a^{\sigma^{3}+1}+a^{\sigma^{5}+1}b^{\sigma^2}=1, \\
b^{\sigma^5+\sigma^3+\sigma}=-a^{\sigma^5+\sigma^2},\\
\N_{q^8/q}(a)=\N_{q^8/q}(b),
\end{cases}
$$
which is equivalent to
\begin{equation}\label{eq:sist3}\begin{cases}
-b^{\sigma^{5}}b^{\sigma^3+\sigma+\sigma^7}-b^{1+\sigma^6+\sigma^4}b^{\sigma^2}=1, \\
b^{\sigma^5+\sigma^3+\sigma}=-a^{\sigma^5+\sigma^2},\\
\mathrm{N}_{q^8/q}(a)=\mathrm{N}_{q^8/q}(b).
\end{cases}
\end{equation}
By multiplying the second equation by $b^{\sigma^7}$ and rewriting the first equation, System \eqref{eq:sist3} becomes
$$\begin{cases}
\N_{q^8/q^2}(b)+\N_{q^8/q^2}(b)^\sigma=-1, \\
a^{\sigma^4+\sigma}b^{\sigma^6}=-\N_{q^8/q^2}(b),\\
\N_{q^8/q^2}(a)=\mathrm{N}_{q^8/q^2}(b).
\end{cases}
$$
We now use that $\N_{q^8/q}(a)=1$ because of Theorem \ref{Gow}, so that the above system can be written as follows
\begin{equation}\label{eq:sist4}
\begin{cases}
\N_{q^8/q^2}(b)+\N_{q^8/q^2}(b)^\sigma=-1, \\
a^{\sigma^4+\sigma}b^{\sigma^6}=-\N_{q^8/q^2}(b),\\
\N_{q^8/q}(b)=\N_{q^8/q^2}(b)^{\sigma+1}=1, \\
\N_{q^8/q}(a)=1.
\end{cases}
\end{equation}

Assume now that $\N_{q^8/q^2}(b) \in \F_q$, the first equation of the above system implies that $q$ cannot be even and $\N_{q^8/q^2}(b)=-\frac{1}{2}$, so that the third equation of System \eqref{eq:sist4} yields $q$ to be a power of $3$ and System  \eqref{eq:sist4} becomes
\[
\begin{cases}
b=\frac{-1}{a^{\sigma^6+\sigma^3}},\\
\N_{q^8/q}(a)=1.
\end{cases}
\]

Suppose that $\N_{q^8/q^2}(b) \notin \F_q$.
By the first and the third equation of System \eqref{eq:sist4} then $\N_{q^8/q^2}(b)$ and $\N_{q^8/q^2}(b)^\sigma$ are two distinct roots of
$$x^2+x+1=0.$$
This equation has two roots not in $\mathbb{F}_q$ if and only if $q\equiv 2 \pmod{3}$ by \cite[(xi) Section 1.5]{Hir}.
So, System \eqref{eq:sist4} admits at least a solution if and only if $q\equiv 2 \pmod{3}$ and the solutions are
\[
\begin{cases}
b=\frac{-\alpha}{a^{\sigma^6+\sigma^3}},\\
\N_{q^8/q}(a)=1.
\end{cases}
\]
This completes the proof.
\end{proof}

\begin{proposition}\label{prop:d=4}
Let $L(x)=ax+bx^{\sigma}-x^{\sigma^4}\in \mathcal{L}_{n,q,\sigma}$ with $\sigma$ a generator of $\mathrm{Gal}(\F_{q^6}\colon\F_q)$.
\begin{itemize}
    \item If $n\leq 11$ and $n\ne 4,8$ then $\dim_{\F_q}(\ker(L(x)))<4$;
    \item if $n\in\{4,8,12\}$ then $\dim_{\F_q}(\ker(L(x)))=4$ if and only if $b=0$ and $\N_{q^n/q^4}(a)=1$;
    \item if $n=13$ then $\dim_{\F_q}(\ker(L(x)))=4$ if and only if
    \[\begin{cases} 
    \N_{q^n/q}(a)=-1,\\
    b=-a^{\sigma(1+\sigma^4+\sigma^8+\sigma^{12})},\\
    q\,\,\text{is a power of}\,\, 3;
    \end{cases}\]
    \item if $n=14$ then $\dim_{\F_q}(\ker(L(x)))=4$ if and only if
    \[ \begin{cases}
    a c_{3,2}=1,\\
    a^{1+\sigma}b^{\sigma^6+\sigma^3}=-b,\\
    a^{1+\sigma+\sigma^2+\sigma^6+\sigma^{10}}=b^{1+\sigma},\\
    c_{3,1}=0,
    \end{cases}\]
where $$c_{3,2}=a^{\sigma^{10}}b^{\sigma^6}b^{\sigma^3}+b^{\sigma^{10}}a^{\sigma^7}b^{\sigma^3}+b^{\sigma^{10}}b^{\sigma^7}a^{\sigma^4},$$
and 
$$c_{3,1}=a^{\sigma^{10}}a^{\sigma^6}b^{\sigma^2}+a^{\sigma^{10}}b^{\sigma^6}a^{\sigma^3}+b^{\sigma^{10}}a^{\sigma^7}a^{\sigma^3}.$$
    \item if $n=15$ then $\dim_{\F_q}(\ker(L(x)))=4$ if and only if
    \[ \begin{cases}
    \N_{q^{15}/q}(a)=1,\\
    b=\frac{1}{a^{\sigma^4+\sigma^8+\sigma^{12}}},\\
    q\,\,\text{is even}.\\
    \end{cases}\]
\end{itemize}
\end{proposition}

\begin{proof}
Again, by applying Theorem \ref{th:mainMcG} we get the first three points. 
Let $n=15$. By Theorem \ref{th:main} $L(x)$ has maximum kernel if and only if 
\begin{equation}\label{eq:sist41}\begin{cases}
b^{\sigma^{11}}a^{\sigma^8}a^{\sigma^4}+a^{\sigma^{11}}b^{\sigma^7}a^{\sigma^4}+a^{\sigma^{11}}a^{\sigma^7}b^{\sigma^3}=a^{-1}, \\
a^{\sigma^{11}}b^{\sigma^{7}}b^{\sigma^4}+b^{\sigma^{11}}a^{\sigma^{8}}b^{\sigma^4}+b^{\sigma^{11}}b^{\sigma^{8}}a^{\sigma^5}=-ba^{-1}a^{-\sigma},\\
b^{\sigma^5+\sigma^8+\sigma^{11}}=b^{\sigma}ba^{-1}a^{-\sigma}a^{-\sigma^2},\\
a^{\sigma^3+\sigma^7+\sigma^{11}}=-b^{\sigma^2}b^{\sigma}ba^{-1}a^{-\sigma}a^{-\sigma^2}. \\
\end{cases}
\end{equation}
By the third and the fourth equations of \eqref{eq:sist41} we get $b^{\sigma^2+\sigma^5+\sigma^8+\sigma^{11}}=-a^{\sigma^3+\sigma^7+\sigma^{11}}$ and, together with its $\sigma$ and $\sigma^2$ power and the fourth equation of \eqref{eq:sist41}, gives $\N_{q^{15}/q}(a)=\N_{q^{15}/q}(b)$.
So, \eqref{eq:sist41} is equivalent to 
$$\begin{cases}
b^{\sigma^{11}}a^{\sigma^8}a^{\sigma^4}+a^{\sigma^{11}}b^{\sigma^7}a^{\sigma^4}+a^{\sigma^{11}}a^{\sigma^7}b^{\sigma^3}=a^{-1}, \\
a^{\sigma^{11}}b^{\sigma^{7}}b^{\sigma^4}+b^{\sigma^{11}}a^{\sigma^{8}}b^{\sigma^4}+b^{\sigma^{11}}b^{\sigma^{8}}a^{\sigma^5}=-ba^{-1}a^{-\sigma},\\
b^{\sigma^2+\sigma^5+\sigma^8+\sigma^{11}}=-a^{\sigma^3+\sigma^7+\sigma^{11}},\\
\N_{q^{15}/q}(a)=\N_{q^{15}/q}(b).
\end{cases}
$$
By replacing the third equation in the first equation of the above system and by Theorem \ref{Gow} we get
\begin{equation}\label{eq:sist42}
\begin{cases}
\N_{q^{15}/q^3}(b)^{\sigma^2}+\N_{q^{15}/q^3}(b)^\sigma+\N_{q^{15}/q^3}(b)=-1, \\
a^{\sigma^{11}}b^{\sigma^{7}}b^{\sigma^4}+b^{\sigma^{11}}a^{\sigma^{8}}b^{\sigma^4}+b^{\sigma^{11}}b^{\sigma^{8}}a^{\sigma^5}=-ba^{-1}a^{-\sigma},\\
b^{\sigma^2+\sigma^5+\sigma^8+\sigma^{11}}=-a^{\sigma^3+\sigma^7+\sigma^{11}},\\
\N_{q^{15}/q}(a)=\N_{q^{15}/q}(b)=-1.
\end{cases}
\end{equation}
By the third equation we get $b=\frac{-\N_{q^{15}/q^3}(b)}{a^{\sigma^4+\sigma^8+\sigma^{12}}}$, so that by replacing such value in the second equation of System \eqref{eq:sist42} we get 
\begin{equation}\label{eq:sist43}
\begin{cases}
\N_{q^{15}/q^3}(b)^{\sigma^2}+\N_{q^{15}/q^3}(b)^\sigma+\N_{q^{15}/q^3}(b)=-1, \\
\N_{q^{15}/q^3}(b)^{2\sigma}+\N_{q^{15}/q^3}(b)^{\sigma+\sigma^2}+\N_{q^{15}/q^3}(b)^{2\sigma^2}=\N_{q^{15}/q^3}(b),\\
b=\frac{-\N_{q^{15}/q^3}(b)}{a^{\sigma^4+\sigma^8+\sigma^{12}}},\\
\N_{q^{15}/q}(a)=\N_{q^{15}/q}(b)=-1.
\end{cases}
\end{equation}
Note that if $\N_{q^{15}/q^3}(b)\in \F_q$, the square of the first equation and the sum of the second equation with its $\sigma$-powers of System \eqref{eq:sist43} implies that $q$ is not a power of $3$ and $\N_{q^{15}/q^3}(b)=-1/3$. From the second equation of System \eqref{eq:sist43} we get $q$ is even and $\N_{q^{15}/q^3}(b)=1$. So, in this case System \eqref{eq:sist43} becomes 
\[
\begin{cases}
b=\frac{1}{a^{\sigma^4+\sigma^8+\sigma^{12}}},\\
\N_{q^{15}/q}(a)=1.
\end{cases}
\]
Now, assume $\N_{q^{15}/q^3}(b)\notin \F_q$, so that $|\{\N_{q^{15}/q^3}(b),\N_{q^{15}/q^3}(b)^\sigma,\N_{q^{15}/q^3}(b)^{\sigma^2}\}|=3$.
If $q$ is not a power of $3$, then by using the first equation of System \eqref{eq:sist43}, the second equation may be replaced by the following
\[ \N_{q^{15}/q^3}(b)^{\sigma+1}+\N_{q^{15}/q^3}(b)^{\sigma^2+1}+ \N_{q^{15}/q^3}(b)^{\sigma^2+\sigma}=1.\]
Hence, System \eqref{eq:sist43} is equivalent to 
\[
\begin{cases}
\N_{q^{15}/q^3}(b)^{\sigma^2}+\N_{q^{15}/q^3}(b)^\sigma+\N_{q^{15}/q^3}(b)=-1, \\
\N_{q^{15}/q^3}(b)^{\sigma+1}+\N_{q^{15}/q^3}(b)^{\sigma^2+1}+ \N_{q^{15}/q^3}(b)^{\sigma^2+\sigma}=1,\\
b=\frac{-\N_{q^{15}/q^3}(b)}{a^{\sigma^4+\sigma^8+\sigma^{12}}},\\
\N_{q^{15}/q}(a)=\N_{q^{15}/q}(b)=-1,
\end{cases}
\]
that is $\N_{q^{15}/q^3}(b),\N_{q^{15}/q^3}(b)^\sigma,\N_{q^{15}/q^3}(b)^{\sigma^2}$ are distinct solutions of
\[ x^3+x^2+x+1=0 \]
not lying in $\F_q$.
Hence the contradiction follows by noting that $-1$ is a root of $x^3+x^2+x+1$.
Let $q$ be a power of $3$. Arguing as before, one have that System \eqref{eq:sist43} is equivalent to
\begin{equation}\label{eq:sist44}
\begin{cases}
\N_{q^{15}/q^3}(b)^{\sigma^2}+\N_{q^{15}/q^3}(b)^\sigma+\N_{q^{15}/q^3}(b)=-1, \\
\N_{q^{15}/q^3}(b)^{\sigma+1}+\N_{q^{15}/q^3}(b)^{\sigma^2+1}+ \N_{q^{15}/q^3}(b)^{\sigma^2+\sigma}\\
\hskip 2 cm =-1+\N_{q^{15}/q^3}(b)^2+\N_{q^{15}/q^3}(b)^{2\sigma}+\N_{q^{15}/q^3}(b)^{2\sigma^2},\\
b=\frac{-\N_{q^{15}/q^3}(b)}{a^{\sigma^4+\sigma^8+\sigma^{12}}},\\
\N_{q^{15}/q}(a)=\N_{q^{15}/q}(b)=-1.
\end{cases}
\end{equation}
By the third and the fourth equations we have 
\[ \N_{q^{15}/q^3}(b)=\frac{-\N_{q^{15}/q^3}(b)^5}{\N_{q^{15}/q^3}(a)^{\sigma^2+\sigma+1}}=\N_{q^{15}/q^3}(b)^5, \]
that is $\N_{q^{15}/q^3}(b)^2=\pm 1$, so that by the second equation of System \eqref{eq:sist44} we have 
\[\N_{q^{15}/q^3}(b)^{\sigma+1}+\N_{q^{15}/q^3}(b)^{\sigma^2+1}+ \N_{q^{15}/q^3}(b)^{\sigma^2+\sigma}=-1.\]
So, $\N_{q^{15}/q^3}(b)$ is a solution of the following system
\[ \left\{ \begin{array}{ll}x^3+x^2-x+1=0,\\ x^2=\pm1, \end{array}\right. \]
from which we get $\N_{q^{15}/q^3}(b)=0$, a contradiction.
\end{proof}

\section{Applications}\label{sec:applications}

\subsection{Rank metric codes}

Rank metric codes were introduced by Delsarte \cite{Delsarte} in 1978 and they have been intensively investigated in recent years because of their applications; we refer to \cite{sheekey_newest_preprint} for a recent survey on this topic.
The set of $m \times n$ matrices $\F_q^{m\times n}$ over $\F_q$ may be endowed with a metric, called \emph{rank metric}, defined by
\[d(A,B) = \mathrm{rk}\,(A-B).\]
A subset $\C \subseteq \F_q^{m\times n}$ equipped with the rank metric is called a \emph{rank metric code} (shortly, a \emph{RM}-code).
The minimum distance of $\C$ is defined as
\[d = \min\{ d(A,B) \colon A,B \in \C,\,\, A\neq B \}.\]
Denote the parameters of a RM-code $\C\subseteq\F_q^{m,n}$ with minimum distance $d$ by $(m,n,q;d)$.
We are interested in $\fq$-\emph{linear} RM-codes, i.e. $\F_q$-subspaces of $\fq^{m\times n}$.
Delsarte showed in \cite{Delsarte} that the parameters of these codes must obey a Singleton-like bound, i.e.
\[ |\C| \leq q^{\max\{m,n\}(\min\{m,n\}-d+1)}. \]
When equality holds, we call $\C$ a \emph{maximum rank distance} (\emph{MRD} for short) code.
Examples of $\F_q$-linear MRD-codes were first found in \cite{Delsarte,Gabidulin}.
We say that two $\F_q$-linear RM-codes $\C$ and $\C'$ are equivalent if there exist $X \in \mathrm{GL}(m,q)$, $Y \in \mathrm{GL}(n,q)$, and $\sigma\in{\rm Aut}(\fq)$ such that
\[\C'=\{XC^\sigma Y \colon C \in \C\}.\]

The \emph{left} and \emph{right} idealisers of $\C$ are defined in \cite{LN2016} as $L(\C)=\{A \in \mathrm{GL}(m,q) \colon A \C\subseteq \C\}$ and $R(\C)=\{B \in \mathrm{GL}(n,q) \colon \C B \subseteq \C\}$. They are invariant under the equivalence of rank metric codes, and have been investigated in \cite{LTZ2}; further invariants have been introduced in \cite{GZ,NPH2}.

When $m=n$, an $\F_q$-linear rank metric code of $\F_q^{n\times n}$ can be seen an $\F_q$-subspace of $\mathcal{L}_{n,q}$ (or $\mathcal{L}_{n,q,\sigma}$) where $d(f(x),g(x))=\mathrm{rk}(f(x)-g(x))=n-\dim_{\F_q}(\ker(f(x)-g(x)))$ for any $f(x),g(x) \in \mathcal{L}_{n,q}$.
In \cite{CsMPZh}, the authors study those rank metric codes of $\mathcal{L}_{n,q}$ with both the idealisers isomorphic to $\F_{q^n}$.
In particular, in \cite[Theorem 2.2]{CsMPZh} they proved that any $\F_q$-linear rank metric code of $\mathcal{L}_{n,q}$ with both the idealisers isomorphic to $\F_{q^n}$ is equivalent to the $\F_{q^n}$-span of some linearized monomials. Then they classify these codes when $n\leq 9$, \cite[Theorem 1.1]{CsMPZh}, introducing a new class of MRD-codes
\[ \C_{3,n}=\langle x,x^q,x^{q^3} \rangle_{\F_{q^n}}, \]
when $n=7$ and $q$ is odd, and $n=8$ and $q\equiv 1 \pmod{3}$.
For $n=7$ the property of $\C_{3,n}$ of being MRD or not may be obtained as a consequence of (c) of Theorem \ref{McGuireMullers=1}, whereas the case $n=8$ may be obtained by using Proposition \ref{prop:d=3}.
Moreover we can determine the number of minimum weight codewords of 
\[\C_{3,n,\sigma}=\langle x,x^\sigma,x^{\sigma^3} \rangle_{\F_{q^n}} \] 
where $\sigma$ is a generator of $\mathrm{Gal}(\F_{q^n}\colon \F_q)$, when $\C_{3,n,\sigma}$ is not an MRD-code and $n \in \{7,8\}$.
Note that the minimum distance of $\C_{3,n,\sigma}$ is greater than or equal to $n-3$ and the number of minimum weight codewords of an MRD-code (that is those of weight $n-2$) depends only on its parameters, see \cite{Ravagnani}.

\begin{corollary}\label{cor:C3n}
Let $q$ be a prime power, $n$ be a positive integer and let $\sigma$ is a generator of $\mathrm{Gal}(\F_{q^n}\colon \F_q)$.
Denote by $D$ the number of codewords of weight $n-3$ in $\C_{3,n,\sigma}$.
\begin{itemize}
    \item  If $n\in \{4,5\}$ then $D=0$, and if $n =6$ then $D=(q^6-1)(q^3+1)$.
    \item If $n=7$, then $D$ is either $0$ when $q$ is odd, or $\frac{(q^7-1)^2}{q-1}$ when $q$ is even. In particular $\C_{3,7,\sigma}$ is an MRD-code if and only if $q$ is odd.
    \item If $n=8$, then $D$ is either $0$ when $q\equiv 1 \pmod{3}$, or $\frac{(q^8-1)^2}{q-1}$ when $q\equiv 0 \pmod{3}$, or $2\frac{(q^8-1)^2}{q-1}$ when $q\equiv 2 \pmod{3}$. In particular $\C_{3,8,\sigma}$ is an MRD-code if and only if $q\equiv 1 \pmod{3}$.
\end{itemize}
\end{corollary}
\begin{proof}
The number of codewords in $\C_{3,n,\sigma}$ having weight $n-3$ coincides with the number of polynomials $\C_{3,n,\sigma}$ with kernel of dimension $3$.
So, by Theorem \ref{th:mainMcG} (see also Theorem \ref{McGuireMullers=1}) the first point follows.
Let $n=7$. By Proposition \ref{prop:d=3}, the polynomials in $\C_{3,7,\sigma}$ with kernel of dimension $3$ exists only when $q$ is even and are of the following form
\[ \eta (x^{\sigma^3}-bx^\sigma-ax), \]
with $\N_{q^7/q}(a)=1$, $b=-a^{\sigma(1+\sigma^3+\sigma^6)}$ and $\eta \in \F_{q^7}^*$.
Similarly, for the case $n=8$ by Proposition \ref{prop:d=3} the polynomials having kernel of dimension $3$ in $\C_{3,8,\sigma}$ exists only if $q \equiv 0,2 \pmod{3}$ and have the following shape
\[ \eta (x^{\sigma^3}-bx^\sigma-ax), \]
with $\N_{q^8/q}(a)=1$, $b=-\frac{\alpha}{a^{\sigma^3+\sigma^6}}$, where $\alpha$ is a root of $x^2+x+1$ and $\eta \in \F_{q^8}^*$.
The assertion then follows by noting that the equation $x^2+x+1=0$ has one root if $q \equiv 0 \pmod{3}$ and it has two distinct roots if $q \equiv 1 \pmod{3}$.
\end{proof}

Now consider $\sigma$ a generator of $\mathrm{Gal}(\F_{q^n}/\F_q)$, $d$ a positive integer with $d\leq n$ and let
\[ \C_{d,n,\sigma}=\langle x,x^\sigma,x^{\sigma^d}\rangle_{\F_{q^n}}. \]
The minimum distance of $\C_{d,n,\sigma}$ is greater than or equal to $n-d$.
Denote by $D$ the number of codewords of weight $n-d$ in $\C_{d,n,\sigma}$.

Similarly to $\C_{3,n}$, as a corollary of Proposition \ref{prop:d=4} we determine the number of codewords in $\C_{4,n,\sigma}$ with weight $n-4$.

\begin{corollary}
Let $q$ be a prime power, $n$ be a positive integer, $\sigma$ be a generator of $\mathrm{Gal}(\F_{q^n}/\F_q)$ and let
\[ \C_{4,n,\sigma}=\langle x,x^\sigma,x^{\sigma^4}\rangle_{\F_{q^n}}. \]
\begin{itemize}
    \item If $5\leq n\leq 11 $ and $n\ne 8$ then $D=0$, and if $n\in\{8,12\}$ then $D=\frac{(q^n-1)^2}{q^4-1}$.
    \item If $n=13$ then 
    \[ D= \left\{ \begin{array}{ll} 0, & \text{if}\,\,q \not \equiv 0 \pmod{3},\\
    \frac{(q^n-1)^2}{q-1}, & \text{otherwise}.
    \end{array} \right. \]
    \item If $n=15$ then 
    \[ D= \left\{ \begin{array}{ll} 0, & \text{if}\,\,q\,\,\text{is odd},\\
    \frac{(q^n-1)^2}{q-1}, & \text{otherwise}.
    \end{array} \right. \]
\end{itemize}
\end{corollary}

Using Theorems \ref{th:mainMcG}, \ref{th:main} and \ref{zullopaultheo} and arguing as in Corollary \ref{cor:C3n} we obtain results on $D$ for $\C_{d,n,\sigma}$.

\begin{corollary}
Let $q$ be a prime power, $n$ be a positive integer, $\sigma$ be a generator of $\mathrm{Gal}(\F_{q^n}/\F_q)$.
\begin{itemize}
    \item If $n\leq d(d-1)$ and $d\nmid n$, then $D=0$.
    \item If $n\leq d(d-1)$ and $d\mid n$, then $D=\frac{(q^n-1)^2}{q^d-1}$.
    \item If $n=d(d-1)+1$ then
    \[ D= \left\{ \begin{array}{ll} 0, & \text{if}\,\,d-1\,\,\text{is not a power of the characteristic of}\,\,\F_q,\\
    \frac{(q^n-1)^2}{q-1}, & \text{otherwise}.
    \end{array} \right. \]
    \item If $n=d^2-1$, $d$ is a power of $2$, $q$ is even, then $D\geq \frac{(q^n-1)^2}{q-1}$.
\end{itemize}
\end{corollary}

\subsection{Quasi-subfield polynomials}

A \emph{quasi-subfield polynomial} is defined as a polynomial of the form $x^{q^d}-\lambda(x) \in \F_{q^n}[x]$ which divides $x^{q^n}-x$ and for which $\log_q(\deg(\lambda(x)))<d^2/n$.
Those polynomials were recently introduced in \cite{HKP} in order to solve the Elliptic Curve Discrete Logarithm Problem (ECDLP). 

As for the family of subspace polynomials introduced in \cite{McGuireMuller}, we show that the polynomials in Theorem \ref{zullopaultheo} with $\sigma\colon x \in \mathbb{F}_{q^n}\mapsto x^q \in \F_{q^n}$ are quasi-subfield polynomials.
 
\begin{proposition} \label{quasisubfield}
 Let $L(x)=x^{q^d}-bx^{q}-ax \in \mathcal{L}_{n,q}$, with $n=(d-1)d+d-1$. If
\begin{itemize}
\item $q$ is a power of 2,
\item $d$ is a power of 2,
    \item $\N_{q^n/q}(a)=1$,
    \item $b=a^{-q^d\frac{q^{d(d-1)}-1}{q^d-1}}$,
\end{itemize} 
then $L(x)$ is a quasi-subfield polynomial.
\end{proposition}

\begin{proof}
Clearly, $\log_q(\deg(\lambda(x)))=1$ and $d^2>n=d^2-1$, so that the condition $\log_q(\deg(\lambda(x)))<d^2/n$ is satisfied. By Theorem \ref{zullopaultheo} $\dim_{\F_q}(\ker(L(x)))=d$, i.e. $L$ splits completely over $\F_{q^n}$ and so it divides $x^{q^n}-x$.
\end{proof}

In \cite{HKP}, the authors developed an algorithm to solve ECDLP over the field $\F_{q^n}$ using a quasi-subfield polynomial.
The complexity of the algorithm presented in \cite{HKP} in which $n=d^2-1$ and with the use of $L(x)=x^{q^d}-bx^q-ax\in \mathcal{L}_{d^2-1,q}$ as in Proposition \ref{quasisubfield}, is less than the complexity of a brute force algorithm, but greater than the generic algorithms (Pollard Rho or Baby-Step-Giant-Step).

See also \cite{Euler,EulerP}.

\subsection{Cyclic subspace codes}

Let $k$ be a non-negative integer with $k \leq n$, the set of all $k$-dimensional $\F_q$-subspaces of $\F_{q^n}$, viewed as $\F_{q}$-vector space, forms a \emph{Grassmanian space} over $\F_q$, which is denoted by $\mathcal{G}_{q}(n,k)$. A \emph{constant dimension subspace code} is a subset $C$ of $\mathcal{G}_{q}(n,k)$ endowed with the metric defined as follows \[d(U,V)=\dim_{\F_q}(U)+\dim_{\F_q}(V)-2\dim_{\F_q}(U \cap V),\]
where $U,V \in C$.
The interest in subspace codes has recently increased because of their application to error correction in random
network coding, see \cite{KoetterK}. 
One of the most studied family of subspace codes was introduced in \cite{Etzion} and is knows as \emph{cyclic subspace codes}.
Let $V \in \mathcal{G}_{q}(n,k)$ and $\alpha \in \F_{q^n}^*=\F_{q^n} \setminus \{0\}$, the \emph{cyclic shift} of $V$ by $\alpha$ is defined as $\alpha V=\{\alpha v:v \in V\}$. Note that a cyclic shift of $V$ is an $\F_{q}$-subspace of $\F_{q^n}$ of the same dimension of $V$. So, a subspace code $C \subseteq \mathcal{G}_q(n,k)$ is said to be \emph{cyclic} if for every $\alpha \in \F_{q^n}^*$ and every $V \in C$ then $\alpha V \in C$. 
Let $V \in \mathcal{G}_q(n,k)$, the \emph{orbit} of $V$ is the set $\{\alpha V: \alpha \in \F_{q^n}^*\}$, and its cardinality is $(q^n-1)/(q^t-1)$, for some $t$ which divides $n$, see e.g.\ \cite[Theorem 1]{Otal}. In particular, every orbit of a subspace $V \in \mathcal{G}_q(n,k)$ defines a cyclic subspace code of size $(q^n-1)/(q^t-1)$, for some $t \mid n $. Assume $k > 1$. 
It is easy to see that a cyclic subspace code generated by an orbit of a subspace $V$ with size $(q^n-1)/(q-1)$ (that is the maximum possible) has minimum distance at most $2k-2$. In \cite{Trautmann}  the following conjecture arises.

\begin{conjecture} \label{conj}
For every positive integers $n,k$ such that $k\leq n/2$, there exists a cyclic code of size $\frac{q^n-1}{q-1}$ in $\mathcal{G}_q(n,k)$ and minimum distance $2k-2$. 
\end{conjecture}

Ben-Sasson et al.\ \cite{BEGR} used subspace polynomials to generate cyclic subspace codes with size $\frac{q^n-1}{q-1}$ and minimum distance $2k-2$. Precisely, they proved that Conjecture \ref{conj} holds true for any given $k$ and infinitely many values of $n$. Such result was then improved in \cite{Otal}. 
However, the relation between $k$ and $n$ was not immediately clear, as $n$ depends on the degree of the irreducible factors of $x^{q^k}+x^q+x$ in $\F_q$. Finally, Conjecture \ref{conj} was solved in \cite{Roth} for most of the cases, by using \emph{Sidon spaces}.

In this section we construct cyclic subspace codes via subspace polynomials with the same parameters of the codes of Conjecture \ref{conj}, in which the relation of $n$ and $k$ is given explicitly, covering some cases of Conjecture \ref{conj} not covered before. 

In order to do this, we need the concept of \emph{subspace polynomial} (see \cite[Definition 2]{BEGR}), that is a monic $q$-polynomial with maximum kernel.

Let $V \in \mathcal{G}_q(n,k)$, the polynomial $P_V(x)=\prod_{v \in V}(x-v)$ is the unique subspace polynomial whose set of roots is $V$. Clearly, if $P(x)$ is a subspace polynomial having $q$-degree $k$ then $P(x)=P_V(x)$, where $V$ is the kernel of $P$. 

\begin{definition}\cite[Definition 3]{BEGR}
Let $V \in \mathcal{G}_q(n,k)$ and $P_V(x)=x^{q^k}+\sum_{j=0}^i\alpha_jx^{q^j}\in \mathcal{L}_{n,q}$, such that $\alpha_i \neq 0$. We define the gap of $V$ as $gap(V)=k-i$.
\end{definition}

The gap of a subspace polynomial plays an important role for determining the distance between two subspace. Indeed, the following holds.

\begin{proposition}\cite[Corollary 2, Lemma 5 and Corollary 3]{BEGR}  \label{prop:gap}
\begin{enumerate}
    \item If $U,V \in \mathcal{G}_q(n,k)$, then $d(U,V) \geq 2\min(gap(U),gap(V))$.
    \item If $V \in \mathcal{G}_q(n,k)$ and $\alpha \in \F_{q^n}^*$, then $P_{\alpha V}(x)=\alpha^{q^k} P_V(\alpha^{-1} x)$. In particular, $V$ and $\alpha V$ have the same gap.
    \item Let $V \in \mathcal{G}_q(n,k)$ and $P_V(x)=x^{q^k}+\sum_{j=0}^i\alpha_jx^{q^j}$. If $\alpha_h \neq 0$ for some $h \in \{1,\ldots,i\}$ with $\gcd(h,n)=t$, then $V$ has at least $\frac{q^n-1}{q^t-1}$ distinct cyclic shifts.
\end{enumerate}
\end{proposition}

Now we are able to exhibit two families of cyclic subspace codes.

\begin{proposition} \label{prop:costrcycl}
Let $k$ be a positive integer such that $k-1$ is a power characteristic of $\F_{q}$. Let $a,b \in \mathbb{F}_{q^n}$, with $n=k(k-1)+1$,
such that $\mathrm{N}_{q^n/q}(a)=(-1)^{k-1}$ and $b=-a^{q f_1}$ where $f_1=\sum_{i=0}^{k-1}q^{ik}$. Let $V$ be the kernel of $P(x)$, with $P(x)=x^{q^k}-bx^q-ax$. Then $C=\{\alpha V: \alpha \in \F_{q^N}^*\}$ is a cyclic subspace code in $\mathcal{G}_q(N,k)$, with size $\frac{q^N-1}{q-1}$ and minimum distance $2k-2$, for every  $N=nh$ with $h \in \mathbb{N}$.
\end{proposition}

\begin{proof}
Let $P(x)=x^{q^k}-bx^q-ax$, with $n,k,a$ and $b$ be as in the assumptions. By (c) of Theorem \ref{McGuireMullers=1}, the polynomial $P(x)$ splits completely over $\F_{q^n}$, so that $P(x)=P_V(x)$. By 3.\ of Proposition \ref{prop:gap}, the code $C=\{\alpha V: \alpha \in \F_{q^n}^*\} \subseteq \mathcal{G}_q(n,k)$ has size $\frac{q^n-1}{q-1}$. Moreover, 1.\ and 2.\ of Proposition \ref{prop:gap}, the minimum distance of $C$ is $2k-2$. Clearly, the result holds in each extension of $\F_{q^n}$.
\end{proof}

Similarly, as a corollary of Theorem \ref{zullopaultheo} and Propositions \ref{prop:d=3} one can prove the following.

\begin{proposition}
Let $k$ and $q$ be powers of $2$. 
Let $n=k^2-1$ and $a,b \in \mathbb{F}_{q^n}$
such that $\mathrm{N}_{q^n/q}(a)=(-1)^{k-1}$ and $b=a^{-\frac{q^{k^2}-q}{q^k-1}}$. Let $V$ be the kernel of $P(x)=x^{q^k}-bx^q-ax$. Then $C=\{\alpha V: \alpha \in \F_{q^N}^*\}$ is a cyclic subspace code in $\mathcal{G}_q(N,k)$, with size $\frac{q^N-1}{q-1}$ and minimum distance $2k-2$, for every  $N=nh$ with $h \in \mathbb{N}$.
\end{proposition}

\begin{proposition}
Let $n=8$, $q\equiv 0,2 \pmod{3}$, $a \in \F_{q^8}$ with $\N_{q^8/q}(a)=1$, $b=\frac{-\alpha}{a^{q^6+q^3}}$, where $\alpha$ is a root of $x^2+x+1$. Let $V$ be the kernel of $P(x)=x^{q^3}-bx^q-ax$. Then $C=\{\alpha V: \alpha \in \F_{q^N}^*\}$ is a cyclic subspace code in $\mathcal{G}_q(N,3)$, with size $\frac{q^N-1}{q-1}$ and minimum distance $4$, for every  $N=8h$ with $h \in \mathbb{N}$.
\end{proposition}

Conjecture \ref{conj} was solved in \cite[Corollary 35.]{Roth} for the following parameters:
\begin{itemize}
    \item Any $q,n$ and $k$ with $k$ up to the largest divisor of $n$ that is smaller than $n/2$;
    \item For any $q \geq 3$ and even $n \geq 2k$.
\end{itemize}
So that the results in \cite{Roth} do not cover the cases in which $n$ is a prime number. 
Therefore, as a consequence of Proposition \ref{prop:costrcycl}, we have the following result.

\begin{corollary}
Let $q=p^h$ and $k=p^\ell+1$, where $p$ is a prime number and $h,\ell \in \mathbb{N}$. If $n=(k-1)k+1$ is a prime then the parameters $(q,n,k)$ of the construction in Proposition \ref{prop:costrcycl} are new solutions of Conjecture \ref{conj}.
\end{corollary}

If $k=p^\ell+1$, then $n=p^{2\ell}+p^\ell+1$. 
Up to our knowledge, it is not known whether or not there infinitely many prime $p$ such that $n$ is a prime. 
This problem is related to a well-known conjecture, see Bunyakovsky conjecture \cite{Boun}.
If $\ell=1$ or $\ell=2$, the values of $n$ are considered in the OEIS sequence \href{https://oeis.org/A053182}{A053182} and \href{https://oeis.org/A066100}{A066100}, respectively, \cite{Seq}.

\section{Conclusions and open problems}\label{sec:final}

In this paper we have provided closed formulas for the coefficients of a $\sigma$-trinomial $f$ over $\F_{q^n}$ ensuring that $f$ has maximum kernel.
As a consequence, we construct explicit examples of $\sigma$-linearized trinomials with maximum kernel and characterize those having $\sigma$-degree $3$ and $4$.
Then we applied these results to rank metric codes, quasi-subfield polynomials and cyclic subspace codes.
We strongly believe that these results may be applied in further contexts.
Indeed, we are currently working on applying these results for the list decodability of certain rank metric codes, see \cite{SZ2}.
However still many open problems remain and we conclude the paper by listing some of them.

\begin{itemize}
    \item By the results contained in \cite{BZ2020}, when $q$ is large enough, there always exists a $\sigma$-linearized polynomial with maximum kernel of form
    \[ -x^{\sigma^d}+a_{d-2}x^{\sigma^{d-2}}+\ldots+a_0x\in \mathcal{L}_{n,q,\sigma}. \]
    Thus, there could be room for $\sigma$-linearized polynomials with maximum kernel and having some coefficients of high $\sigma$-degree equal to zero.
    \item When $n=d^2-1$, there are others $\sigma$-linearized trinomials with maximum kernel different from those found in Theorem \ref{zullopaultheo}, as suggested by Proposition \ref{prop:d=3}. Determine the remaining family of $\sigma$-linearized polynomials with maximum kernel when $n=d^2-1$.
    \item More generally, is it possible to manipulate the equations in Theorem \ref{th:main} to get some more examples of linearized trinomials with maximum kernel also when $g \in \{2,\ldots,d-2\}$ and $d\geq 4$?
\end{itemize}

\section{Acknowledgements} 

We would like to thank Professor Robert Israel for pointing out Bunyakovsky conjecture and the OEIS sequences A053182 and A066100. 
We also thank the anonymous referee and Bence Csajb\'ok, which helped us to significantly shorten the proof of Theorem \ref{th:mainMcG}.
This research was supported by the Italian National Group for Algebraic and Geometric Structures and their Applications (GNSAGA - INdAM).
The second author is also supported by the project ``VALERE: VAnviteLli pEr la RicErca" of the University of Campania ``Luigi Vanvitelli''.

\bigskip

\noindent Paolo Santonastaso and Ferdinando Zullo\\
Dipartimento di Matematica e Fisica,\\
Universit\`a degli Studi della Campania ``Luigi Vanvitelli'',\\
Viale Lincoln, 5\\
I--\,81100 Caserta, Italy\\
{{\em \{paolo.santonastaso,ferdinando.zullo\}@unicampania.it}}

\end{document}